\documentclass{elsarticle}

\usepackage{hyperref}

\usepackage{amssymb,amsmath,amsfonts}

\newtheorem{theorem}{Theorem}

\newtheorem{corollary}[theorem]{Corollary}
\newtheorem{lemma}[theorem]{Lemma}

\newtheorem{definition}[theorem]{Definition}
\newtheorem{question}[theorem]{Question}
\newtheorem{example}{Example} 
\newtheorem{remark}{Remark} 
\newtheorem{proposition}[theorem]{Proposition}

\begin{document}

\begin{frontmatter}

\title{  Set-theoretic solutions of the Yang-Baxter equation and new classes of  R-matrices}

\author[ed]{ Agata Smoktunowicz \corref{cor1} }
\ead{A.Smoktunowicz@ed.ac.uk}

\author[wtech]{Alicja Smoktunowicz \corref{cor2} }
\ead{smok@mini.pw.edu.pl}

\cortext[cor1]{Principal corresponding author}
\cortext[cor2]{Corresponding author}

\address[ed]{School of Mathematics, University of Edinburgh,  Edinburgh EH9 3JZ, Scotland, United Kingdom}  

\address[wtech]{ Faculty of Mathematics and Information Science, Warsaw University of Technology, Koszykowa 75, 00-662 Warsaw, Poland}

\begin{abstract} 
   We describe several  methods of constructing $R$-matrices that are dependent upon many parameters, for example unitary $R$-matrices 
 and  R-matrices whose entries are functions.
 As an application, we construct examples of $R$-matrices with prescribed singular values.
    We characterise some classes of  indecomposable  set-theoretic solutions of the  quantum Yang-Baxter equation (QYBE) and construct  $R$-matrices related to such solutions.  In particular, we  establish a correspondence between one-generator braces and indecomposable, non-degenerate involutive set-theoretic solutions of the  
 QYBE, showing that such solutions are abundant.   
 We show that $R$-matrices related to involutive, non-degenerate solutions of the QYBE have special form.
 We also investigate some linear algebra questions related to $R$-matrices. 
\end{abstract}


\end{frontmatter}

\section{Introduction} 
 The quantum Yang-Baxter equation is an important equation in mathematics and physics. It is relevant to  statistical mechanics, quantum information science and  numerous other research areas.  
Recall that a nonsingular $n^{2}\times n^{2}$ matrix is called  an $R$-matrix when it satisfies the  quantum Yang-Baxter equation:
\[(R\otimes I )(I\otimes R)(R\otimes I) = (I \otimes R)(R \otimes I)(I \otimes R),\]
 where $I$ is the  $n\times n$ identity matrix. 
 In this paper, we will provide examples of $R$-matrices with prescribed singular values. Many of our $R$-matrices are constructed using set-theoretic solutions of the quantum Yang-Baxter equation, and have only one nonzero element in each column.
 This type of matrices appear frequently  in the  literature, for example in \cite{dancer}, where Baxterisation of some $R$-matrices of this type was obtained, and in \cite{gates}, where they appear as  universal gates  for the quantum computation related to the circuit model 
(see Theorem $1$, \cite{gates}). In  \cite{ns, gr, franko}  they appear in  connection with Braid groups and  topological quantum computation (see also   \cite{bn, vk, lr, nikita}).
 They also appear as {\em combinatorial $R$-matrices } in the theory of crystal bases, geometric crystals and box-bell systems. The $R$-matrices related 
  to involutive, non-degenerate set-theoretic solutions of the QYBE 
 give  cocycles into abelian groups \cite{etingof}, therefore  they  can be given as an input in the construction of universal $R$-matrices and twists for Hopf algebras, for example as in Theorem $4.2$,  \cite{gelaki}. 
  
  Most of the   $R$-matrices constructed in our paper are unitary. 
  A unitary R-matrix leads to a unitary representation of the Braid group, and the
resulting unitary matrices associated to braids can be used to process quantum
information \cite{ns, Chen, franko}.  In connection with  the topological quantum computation,  it was conjectured in \cite{gr, rz} that a single unitary $R$-matrix can generate only finite representations of Braid groups, and in \cite {gr} it was confirmed in several important  classes of $R$-matrices.
In \cite{Rowell},  Rowell made the following comment: ``From the quantum information point of view, a unitary R-matrix can be used to directly simulate topological quantum computers on the quantum circuit model.
More recently people have begun to study what extra gates one needs to supplement braiding with in order to achieve universality (see e.g. \cite{ssz}).  
 If a single R-matrix can only generate a (nearly) finite group, can an additional small gate lead to a universal gate set?''.
This question provides the inspiration for our construction of R-matrices with many parameters. All of our examples are locally monomial  BVS (we recall the definition in Section $2$).

  Recall that  locally monomial braided vector spaces (BVS) were introduced  by Galindo and Rowell  in \cite{gr}, and localisation was introduced by Rowell and Wang  in  Definition $2.3$ in  \cite{rz}.
   A related notion of braided vector space of set-theoretic type appeared in \cite{andrusz} in the context of Nichols algebras (note that every braided vector space of set-theoretic type is locally monomial).  In \cite{ghr}, Galindo, Hong and Rowell generalized the idea of localisation in two ways, and
 in \cite{gr} Galindo and Rowell remarked that it is feasible that the monomial and Gaussian
BVS generate a large proportion of  the unitary braided vector spaces (e.g. through quotients and subrepresentations). They also mentioned that they are not aware of any unitary braided vector spaces that do not come from these two
families. Recall that Gaussian braided vector spaces were introduced  28 years ago by Goldschmidt  and Jones, and since then have been investigated by many authors (see for example \cite{gr, rowell2}).
 Since only these two classes of unitary $R$-matrices are known so far, and the Gaussian class is  well understood,  it seems natural to investigate methods of construction of locally monomial BVS; it is the primary motivation of our paper. We will mainly investigate  unitary $R$-matrices related to locally monomial  BVS which are constructed using braces. Recall that braces were introduced by Rump \cite{rump} in 2005. 
 
    The contents of the chapters are as follows:  Section \ref{11111}  contains background information.  Section \ref{567567} investigates linear algebra and matrix theory aspects  related to  $R$-matrices. Section \ref{555} gives new  examples of unitary $R$-matrices constructed by using set-theoretic solutions of the QYBE (all of the examples  are locally monomial).   Section \ref{555} also describes some general methods of constructing unitary locally monomial BVS  by using orbits by analogy with 
 cohomology of racks and cycle sets considered in \cite{cjks, lebed}, and  by introducing $I$-retraction, as a generalisation of the retraction technique from \cite{etingof}.
In the final section we  give concrete examples of $R$-matrices of small dimension illustrating results obtained in this paper.
 
 An additional motivation for our paper is related to  the well-known characterisation of indecomposable solutions of prime order, obtained in  \cite{guralnick} and \cite{etingof}. Recall that in \cite{guralnick},
  Etingof, Guralnick and Soloviev showed that  all indecomposable non-degenerate 
   set-theoretic solutions of the QYBE of prime cardinality are affine, and in  \cite{etingof}, Etingof, Schedler and Soloviev showed that all indecomposable 
 involutive non-degenerate solutions of the QYBE of prime cardinality  are the cyclic permutation solutions. 
    In Sections \ref{ind} and \ref{nil2} we investigate whether it is possible to obtain an analogous  
     characterisation of  indecomposable solutions of arbitrary cardinality.
    We show that every one-generator brace yields a non-degenerate, involutive,  indecomposable solution of the QYBE, and for multipermutation solutions all indecomposable solutions are of this form. We also show that one-generator braces and hence non-degenerate, involutive, indecomposable solutions of the QYBE  are abundant.  This gives a strong indication that it would be impossible to generalise the results from  \cite{guralnick}  and \cite{etingof} to arbitrary cardinalities. As an application, in Section \ref{33333} we  use our results to construct unitary $R$-matrices related to indecomposable solutions.
    In Section \ref{singular}, we give a simple criterion for whether a given $R$-matrix is related to a brace; this observation  relies on the result of Jespers and Okni{\' n}ski that  finite involutive, set-theoretic solutions of the QYBE are left non-degenerate if and only if they are right non-degenerate.

   In integrable systems a different type of the quantum Yang-Baxter equation is used, called the parametrized quantum Yang-Baxter equation (see pages 295-297, \cite{klimyk}). Often, the following  form of this equation  is used $(R(u)\otimes I)(I\otimes R(u + v))(R(v)\otimes I) = (I\otimes R(v))(R(u + v)\otimes I)(I\otimes R(u))$ where $u$ and $v$ are complex variables. 
     Many of the examples of  $R$-matrices obtained in our paper satisfy $R^{2}=I$.  By using  analogous methods as on page 296 \cite{klimyk} it can be shown that  
   if $R$ is an $R$-matrix such that  $R^{2}=I$ then $I+uR$ is a solution of the above parametrized QYBE  (see Proposition \ref{parameterdependent}). Applications of such $R$-matrices are discussed in Section 8.7.3 \cite{klimyk}.
   
  Observe also that, by combining Proposition $2.9$ from \cite{car} (and its proof) with  Examples \ref{hura5}, \ref{simple} from our paper, we can construct extensions of set-theoretic solutions of the QYBE, in particular we can embed involutive solutions in non-involutive solutions of the QYBE. In  \cite{vendramin, squarefree, lebed} extensions of set-theoretic solutions have been used to construct special types of involutive solutions.

  In this paper we use the connection  of non-degenerate, involutive  set-theoretic solutions with nilpotent rings and braces discovered by Rump in 2007 \cite{rump}. A related concept of $F$-braces, introduced in \cite{cr}, has recently found applications in cryptography \cite{codes}. In \cite{ccs} examples of $F$-braces have been constructed by using $2$-cocycles.

\section{ Background information }\label{11111}
 Here we recall  basic information about braces, skew braces, locally monomial braided vector spaces 
 and indecomposable solutions of the  quantum Yang-Baxter equation.

We say that  $X \in \mathbb C^\mathrm{n^2 \times n^2}$ satisfies the quantum Yang-Baxter Equation (QYBE) if
\begin{equation}\label{eqs1}
(X \otimes I_n)    (I_n \otimes X)   (X \otimes I_n)   = (I_n  \otimes X)    (X \otimes I_n)   (I_n \otimes X), 
\end{equation}
where $I_n$ denotes the $n \times n$  identity matrix, and  $A \otimes B$ is the Kronecker product (tensor product) of the matrices $A$ and $B$:
$A \otimes  B= (a_{i,j} B)$. That is, the  Kronecker product $A \otimes  B$ is a block matrix whose $(i,j)$ blocks are $a_{i,j} B$.

Solutions of the quantum Yang-Baxter equation  (\ref{eqs1})   have many interesting properties:
\begin{itemize}
\item If $X \in \mathbb C^\mathrm{n^2 \times n^2}$  satisfies the QYBE (\ref{eqs1}), then $X^{*}$ also satisfies  (\ref{eqs1}).
\item  If $X \in \mathbb C^\mathrm{n^2 \times n^2}$ is an {R-matrix}, then $X^{-1}$ is an {R-matrix}.
\item  If $X \in \mathbb C^\mathrm{n^2 \times n^2}$  satisfies  the QYBE (\ref{eqs1}), then   $ \alpha \, X$  satisfies  the QYBE  (\ref{eqs1})  for every $ \alpha \in \mathbb C$.
\item If $X \in \mathbb C^{n^2 \times n^2}$  satisfies the QYBE  ( \ref{eqs1}) and $P \in \mathbb C^{n\times n}$  is arbitrary nonsingular matrix, then \[\hat X=  ( P \otimes P) X  (P \otimes P)^{-1}\]  also  satisfies  the QYBE (\ref{eqs1}).
\end{itemize}
 The references to the above properties  can be found,  for example,  in \cite{jenfranko, 22}.

 However, it is not true  that if $X \in \mathbb C^{n^2 \times n^2}$ is an   {R-matrix} and  $P, Q \in \mathbb C^{n\times n}$, $P\neq Q$ 
 are arbitrary nonsingular matrices,  then $\bar  X=  ( P \otimes Q) X  (P \otimes Q)^{-1}$
is an {R-matrix}.

 \subsection{Set-theoretic solutions}\label{1234}
 
  Let X be a non-empty set. Let $r: X\otimes X\rightarrow X\otimes X$ be a bijective  map and write
 \[r(x, y) =(\sigma _{x}(y), \tau_{y}(x)).\] We say that $(X, r)$ is a set-theoretic solution of the quantum Yang-Baxter equation if
 \[r_{1}r_{2}r_{1} = r_{2}r_{1}r_{2},\]  where $r_{1} = r\times id_{X} : X\times X\times X \rightarrow X\times X\times X$ and   $r_{2} = id_{X}\times r : X\times X\times X\rightarrow X\times X\times X$.
 We say that $(X,r)$ is  right non-degenerate  if $\sigma _{x} \in  Sym (X)$, for all  $x \in X$, (Sym (X) denotes the set of all permutations of the   set $X$); similarly $(X,r)$ is left non-degenerate if $\tau _{x} \in  Sym (X)$, for all  $x \in X$.
We say that $(X,r)$ is a  non-degenerate involutive set-theoretic
 solution of the quantum Yang-Baxter equation if
 $r^{2} = id_{X\times X}$ 
and  $\sigma _{x}, \tau _{x} \in  Sym (X)$, for all  $x \in X$.
 
 Let $V$ be the linear space spanned by the elements of $X$ over the field of complex numbers.
  By $ \bar {r}:V\otimes V\rightarrow V\otimes V$ we will denote the {\em linearisation } of $r$, i.e. the linear map 
 such that   \[\bar {r} (x\otimes y) =\sigma _{x}(y)\otimes  \tau_{y}(x).\]
 
 Let $(X, r)$ be a set-theoretic solution of the  quantum Yang-Baxter equation, then $(V, \bar{r})$ is a solution of the QYBE. 

 Let $(X,r)$  be a non-degenerate solution of the QYBE. 
 We say that $(X,r)$ is {\em decomposable} if there exist non-empty subsets $X_{1}, X_{2}\subseteq X$ such that  $X=X_{1}\cup X_{2}$ and  $r(X_{i}, X_{j})= (X_{j}, X_{i})$ for all $i,j\leq 2$. 
If it is not possible to find such subsets $X_{1}, X_{2}\subseteq X$  the solution $(X,r)$ is {\em indecomposable}.
In \cite{etingof} Etingof, Schedler and Soloviev  proved that a finite non-degenerate solution $(X,r)$  is indecomposable if and only if $X$ cannot be presented as a union of two non-empty sets $Y_{1}, Y_{2}$ such that $r(Y_{1}, Y_{1})=(Y_{1}, Y_{1})$ and $r(Y_{2}, Y_{2})=(Y_{2}, Y_{2})$.    
 Let $z\in X$. By the {\em orbit } of $z$ we will mean the smallest set $Y\subseteq X$ such that $z\in Y$ and $\sigma _{x}(y)\in Y$ and $\tau _{x}(y)\in Y$, for all $y\in Y, x\in X$.

 \subsection{Locally monomial BVS}\label{123}

 Let $V$ be a linear space over a field $F,$  where,  unless otherwise specified, $F=\mathbb {C}$. 
 Let $I_{V}:V\rightarrow V$ be the identity map on $V$.  
 Recall that a linear authomorphism  
 $c: V\otimes V\rightarrow V\otimes V$ satisfies the quantum Yang-Baxter equation if
 \[(c\otimes I_{V} )(I_{V}\otimes c)(c\otimes I_{V}) = (I_{V} \otimes c)(c \otimes I_{V})(I_{V} \otimes c).\]
 In this case the pair $(V, r)$ will be called a braided vector space (BVS).

Let $(X,r)$ be a set-theoretic solution of the QYBE, and denote as usually
 \[r(x, y) =(\sigma _{x}(y), \tau_{y}(x)).\]

 \begin{definition}\label{set}[Definition $5.8$, \cite{andrusz}] Let $(X,r)$ be a set-theoretic solution of the QYBE and let $D=\{d_{(x, y)}\}_{x,y\in X}$, where $0\neq d_{i,j}\in \mathbb C$.  Let $V$ be the linear space over $\mathbb C$ spanned by elements of $X$.We say that the data $(X,r^{D})$ is {\bf  a braided vector space of set-theoretic type} if 
 the linear map  $r_{D}:V\otimes V\rightarrow V\otimes V$ defined by  \[ {r}_{D}(x\otimes y)=d_{(x,y)}\sigma _{x}(y)\otimes  \tau_{y}(x)\]  satisfies the QYBE.
  Given $(X,r)$ and $D$ by $(X,r^{D})$, we will always mean the linear space $V$ and map $r^{D}$ as above. 
 \end{definition}

Let $(X,r)$ be a set-theoretic solution of the QYBE. In the remainder of this subsection we will also use the following notation 
\[r(x,y)=({ }^xy, x^{y})\] so ${ }^xy=\sigma _{x}(y)$ and $x^{y}=\tau _{y}(x)$.
  We will use it because it is easier to read in Lemma \ref{identity}.
  This notation appears in many papers, for example \cite{tatiana}.
 
 We recall a special case of Lemma $5.7$ from \cite{andrusz}:
\begin{lemma}\label{identity} [Lemma $5.7$, \cite{andrusz}] Let $(X, r)$ be a  non-degenerate, set-theoretic solution of the quantum Yang-Baxter equation. 
 Let $f:X\otimes X\rightarrow \mathbb C$ be a  mapping, and assume that $f$ takes only nonzero values. Then the following statements are equivalent:
\begin{itemize}
\item For every $x,y, z\in X$  
 \[f(x,y)\cdot f(x^{y}, z)\cdot f({{ }^xy},  { }^{x^{y}}z)=f(y,z)\cdot f(x, { }^yz)\cdot f(x^{{ }^yz}, y^{z}).\]
\item  The linear  mapping $c:V\otimes V\rightarrow V\otimes V$ given by 
$c(x\otimes y)=f(x,y){ }^xy{\otimes x^{y}}$ 
 satisfies the quantum Yang-Baxter equation.
\end{itemize}
\end{lemma}
  
   The $R$-matrix of a braided vector space $(X,r^{D})$ of set-theoretic type is obtained in the following way  (see \cite{twoKens} pages $94, 95$).
\begin{definition}\label{monmat}
 Let  $(X,r^{D})$ be a braided vector space of set-theoretic type, where $(X,r)$ is a set-theoretic solution of the QYBE and $D=\{d_{(x,y)}\}_{x,y\in X}$. 
  Let $X=\{x_{1}, \ldots , x_{n}\}$ and let $Y$ be the set of pairs $({i}, {j})$, with $1\leq i,j\leq n$, written in the lexicographical ordering.  Then $M$  has rows and columns indexed by the set $Y,$ and the entry  at the intersection of the column indexed by $(i,j)$ 
and the row indexed by $(k,l)$ equals $d_{(x_{i}, x_{j})}$ if $r(x_{i}, x_{j})=(x_{k}, x_{l})$ and $0$ otherwise. This entry will be denoted by $m_{i,j}^{k,l}$.
\end{definition}    

 The method of writing the $R$-matrix related to a  set-theoretic solution $(X, r)$ of the QYBE 
  is the same, as for $(X, r^{D})$ where all elements $d_{x,y}$ in $D$ are $1$.

We now recall the  definition of {\em locally monomial BVS} from \cite{gr}.
\begin{definition}\label{matrices} 
Let $V$ be a vector space over a field $F$. 
 Let $(V, c)$ be a braided vector space. We say that $(V,r)$ is a {\bf locally monomial BVS} if there is a basis $X=\{x_{1}, \ldots , x_{n}\}$ of $V$ such that 
 for each $i,j\leq n$ there are $k,l\leq n$ and $d_{i,j}\in F$ such that  $c(x_{i}\otimes x_{j})=d_{i,j}(x_{k}\otimes x_{l})$. Note that, with respect to the  base $X$, $(V,r)$ is a braided vector space of set-theoretic type $(X, r^{D})$ for some bijective map $r:X\times X\rightarrow X\times X$ and some $D=\{d_{(x,y)}\}_{x,y\in X}$.
\end{definition}
 Let $(V,c)$ be a braided vector space, then the matrix of the map $c:V\otimes V\rightarrow V\otimes V$  in the  base $V$ will be called a {\em locally monomial $R$-matrix}. 
  Note that 
  an $n^{2}$ by $n^{2}$  $R$-matrix $A$ is {\em locally monomial}  if $A=(P\otimes P)B(P^{-1}\otimes P^{-1})$ for some nonsingular $n$ by $n$ matrix $P$ and where $B$ is an $R$-matrix of some braided vector space  $(X, r^{D})$ of set-theoretic type.
   Notice that $B$ is obtained by modifying nonzero entries of a permutation matrix.     
  
 \begin{definition}\label{triv} Let $(X,r^{D})$ be a braided vector space of set-theoretic type. We say that   
  $(X, r^{D})$ is trivial if there
 are nonzero complex numbers $\alpha _{x}$ for $x\in X$ and a constant $c$  such that \[d_{(x,y)}=c\cdot \alpha _{x}\alpha _{y}(\alpha _{\sigma _{x}(y)})^{-1}(\alpha _{\tau _{y} (x)})^{-1}.\]
 Notice that this is equivalent to the condition that there is a nonsingular diagonal  $n$ by $n$ matrix $P$ and a constant $c$  such that $ \bar {M}=c\cdot (P^{-1}\otimes P^{-1})M(P\otimes P)$, where 
 $M$ is the matrix associated to  the solution $(X,r)$ (as below Definition \ref{monmat}) and $\bar M$  its matrix  related to braided vectors space of set-theoretic type  $(X, r^{D})$ (as in Definition \ref{monmat}).  
\end{definition}
 
\begin{lemma}\label{trivial}  Let $(X,r)$ be an involutive set-theoretic solution of the QYBE, and let $(X, r^{D})$ be a braided vector space  of set-theoretic type; then $d_{(x,y)}d_{(\sigma _{x}(y), \tau _{y}(x))}=c^{2}$ for all $x,y\in X$ for some constant $c$ (this can be also written as  $d_{(x,y)}d_{r(x,y)}=c^{2}$).  
\end{lemma}
 \begin{proof} It follows from Definition \ref{triv} and the fact that $r(r(x,y))=(x,y)$  since $r$ is involutive. 
 \end{proof} $\square $
      \begin{lemma}\label{inna} Let $(X,r)$ be a set-theoretic solution of the QYBE. 
     Let $f:X\times X\rightarrow \mathbb C$ be a  mapping such that, for all $x,y,z\in X$,
 \[f(x,y)=f(x^{{ }^yz}, y^{z}),    f(x^{y}, z)=f(x, { }^yz), 
f({{ }^xy},  { }^{x^{y}}z)=f(y,z).\]
     Denote $D=\{d_{(x,y)}\}_{\{x,y\in X\}}$, where $d_{(x,y)}=f(x,y)$ for $x,y\in X$; then $(X, r^{D})$ is a braided vectors space of set-theoretic type.  
 If $(X,r)$ is a non-degenerate  involutive solution of the QYBE, and $f(x,y)f({ }^xy, x^{y})$ is not constant for $x,y\in X$, then 
  $(X,r^{D})$ is not trivial. 
  \end{lemma}
   \begin{proof}
   It follows from  Lemmas \ref{identity} and \ref{trivial}.
   \end{proof} $\square $
 \subsection{ Braces and skew braces}\label{braces}
 We now recall some basic information about braces and skew-braces.  
   The research area started around 2005, when Wolfgang Rump showed some surprising connections between Jacobson radical rings and solutions to the quantum Yang-Baxter equation. In 
   \cite{rump2}, Rump introduced braces, a generalisation of Jacobson radical rings, as a tool to investigate non-degenerate, involutive set-theoretic solutions of the  quantum Yang-Baxter equation, and showed the correspondence between such solutions and braces. Skew braces were recently introduced by Guarnieri and Vendramin to investigate set-theoretic solutions of the QYBE which are not involutive.
 In \cite{rump} Rump showed that every solution $(X,r)$ can be in a good way embedded in a brace.
 \begin{definition}[Proposition $4$, \cite{rump}]
 A {\em left brace } is an abelian group $(A; +)$ together with a 
 multiplication $\cdot $ such that the circle operation $a \circ  b =
 a\cdot b+a+b$ makes $A$ into a group, and $a\cdot (b+c)=a\cdot b+a\cdot c$.
\end{definition} 
 In many papers, the following equivalent definition from \cite{cjo} is used:
 \begin{definition}[\cite{cjo}]
 A {\em left brace } is a set $G$ together with binary operations $+$ and $\circ $ such that 
 $(G, +)$ is an abelian group, $(G, \circ )$ is a group, and 
  $a\circ (b+c)+a=a\circ b+a\circ  c$ for all $a,b,c\in G$.
\end{definition} 
  The additive identity of a brace $A$ will be denoted by $0$ and the multiplicative identity by $1$. In every brace $0=1$. The same notation will be used for skew braces (in every skew brace $0=1$).
  Let $A$ be a left brace. The {\em socle} of $A$ is
 $Soc(A)=\{a\in A: a\circ b=a+b$  for all $b\in A\}=\{a\in A: a\cdot b=0$  for all $b\in A\}.$
 
 \begin{remark} 
 Some authors use the notation $\cdot $  instead of $\circ $ and $*$ instead of $\cdot $ (see for example \cite{cjo, tatiana}).
 \end{remark}
 It was observed by Rump that every nilpotent ring is a brace.
 Let $R$ be a nilpotent ring (associative, and not necessarily commutative) and let $n$ be such that  $R^{n}=0$. It was shown by Rump \cite{rump} that  $R$ yields a 
 solution $r : R\times R\rightarrow R\times R$  to
 the  quantum Yang-Baxter equation with  $r(x, y) = (u, v)$, where $u=x \cdot y +y$, $v=z \cdot x+x$ and $z=\sum_{i=1}^{n}(-1)^{i}u^{i}$.

Let $R$ be either a left brace or a ring, and let  $C,D\subseteq R$; then $CD$ denotes the set consisting of finite sums of elements $cd$ with $c\in C, d\in D$. 
\begin{definition}
 A left brace $R$ is left nilpotent if  $R^{n}=0$ for some $n$ where $R^{1}=R$ and $R^{i+1}=R\cdot R^{i}$. A left brace is right nilpotent if    $R^{(n)}=0$ for some $n$ where $R^{(1)}=R$ and $R^{(i+1)}=R^{(i)}\cdot R$. Radical chains $R^{i}$ and $R^{(i)}$ were introduced by Rump in \cite{rump}. 
\end{definition}

\begin{theorem}\label{1255}\cite{rump}  Let $(X, r)$ be a  non-degenerate, involutive, set-theoretic solution of the QYBE. Then there exists a left  brace $A$ with multiplication $\cdot $ and addition $+$ such that $X\subseteq A$ and 
 \[r(x,y)=(x\cdot y+y, z\cdot x+x),\] for $x,y\in A,$  
  where $z\cdot (x\cdot y+y)+z+x\cdot y+y=0$ for $x,y\in A$   (in each left brace $A$ such $z$ exists and is unique). Moreover, the groups $(A, \circ )$ and $(A, +)$  are generated by elements from $X$.
  \end{theorem}
  \begin{proposition}\cite{CGIS}\label{1256}
   Let notation be as in Theorem \ref{1255}, if $X$ is finite then $A$ can be chosen to be finite.
 \end{proposition}
 The fact that $(A, \circ )$ is generated by $X$ follows because $A$ is a factor of the structure group of $X$
  \cite{CGIS, cjo}. 
  The structure group of a non-degenerate set-theoretic solution $(X,r)$  is the  group generated by elements of $X$ subject to all relations $xy=uv$ where $r(x,y)=(u,v)$ (this group was introduced in \cite{etingof}). 
  The permutation group of a solution $(X,r)$ is the group generated by mappings $\sigma _{x}$ where $r(x,y)=(\sigma _{x}(y), \tau _{y}(x))$.
 The permutation group first appeared as a tool in the proof of  Theorem $2.15$ in \cite{etingof}, without a concrete name but with the notation 
 $G_{X}^{0}$. In \cite{GI4} the name  {\em permutation group} was introduced and this group was explicitly investigated.

\begin{definition} Let $A$ be a brace, and for $x, y\in A$ define 
 \[r'(x,y)=(x\cdot y+y, z\cdot x+x),\] where  
  $z$ is such that $z\cdot (x\cdot y+y)+z+x\cdot y+y=0$ for $x,y\in A$.
We will say that $r': A\times A\rightarrow A\times A$ is the {\bf Yang-Baxter map associated to $A$}.
Let $X\subseteq A$; we will say that $r:X\times X\rightarrow X\times X$ is a {\bf restriction of $r'$ } if $r(x,y)=r'(x,y)$ for all $x, y\in X$.   
\end{definition}

 In \cite{etingof}, Etingof, Schedler and Soloviev introduced  the retract relation for any  solution $(X,r)$. Denote  $X=\{x_{1}, \ldots , x_{n}\}$ and
 $r(x,y)=(\sigma _{x}(y), \tau _{y}(x))$.  
  Recall that the retract relation $\sim $ on $X$  is defined by $x_{i}\sim   x_{j}$  if
$\sigma _{i} = \sigma _{j}$. The  induced solution
 $Ret(X, r) = (X/\sim , r^{\sim })$  is called the {\em retraction } of $X$. A solution $(X, r)$ is
called  a {\em  multipermutation solution of level m} if m is the smallest nonnegative
integer such that  after $m$ retractions we obtain the solution with one element.
 By $mpl(X,r)$ we will denote the multipermutation level of $(X,r)$. 

 We now recall the definition of a skew brace. Skew braces also yield non-degenerate solutions of the QYBE (see \cite{gv}). 
 \begin{definition}\cite{gv} 
 A {\em skew brace } is a set $G$ together with binary operations $+$ and $\circ $ such that 
 $(G, +)$ is a group, $(G, \circ )$ is a group, and 
  $a\circ (b+c)=a\circ b+(-a)+a\circ  c$ for all $a,b,c\in G$.
 Here $(-a)$ is the inverse of $a$ in the additive group of $G$, so $a+(-a)=0$.
\end{definition} 
 Notice that the group $(A, +)$ need not be commutative; for this reason some authors prefer to use the notation $\cdot $ instead of $+$ in the definition of skew braces \cite{gv, sv}. 
 
\section{ Matrix theory observations}\label{567567}
 
In this section we consider the Hadamard product of $A, B  \in \mathbb C^\mathrm{m \times n}$: $A \circ  B= (a_{i,j} b_{i,j}) $.

\begin{proposition}\label{four} Let $A,B$ be $R$-matrices, such that when we put all nonzero entries of these matrices to $1$ the obtained
 matrix is the same for both matrices  $A$ and $B$, and it is a permutation matrix. Then the Hadamard product of $A$ and $B$ is also an $R$-matrix. 
\end{proposition}
\begin{proof} $A$ is an $n^{2}$ by $n^{2}$ matrix 
 for some $n$. We can assume that the rows and columns of $A$  are indexed by the set of pairs $(i,j)$ with $i,j\leq n$ in the lexicographical ordering. Denote $X=\{1, \ldots , n\}$. Recall that  $A$ satisfies QYBE, is non-singular and has exactly one nonzero-element in each column. It follows that 
  $A$ can be obtained as in Definition \ref{monmat} from some braided vectors space of set-theoretic type $(X, r^{D})$  
  for some  $D=\{f(x,y)\}_{x,y\in X}$. Denote by $a_{i,j}^{k,l}$ the entry at the intersection of the $(i,j)$-th column and $(k,l)$-th row of $A$. By Definition \ref{monmat},  
   $r(i,j)=(k,l)$ if and only if $a_{i,j}^{k,l}\neq 0$.  
   Denote $r(x,y)=({ }^xy, x^{y})$. Since $A$ is an $R$-matrix, the linear mapping $c(x\otimes y)=f(x,y){ }^xy\otimes x^{y}$ satisfies QYBE where $f(x,y)=a_{x,y}^{r(x,y)}$. 
   
   Similarly, $B$ is the matrix related to some braided vector space of set-theoretic type $(X, {\tilde r}^{D'})$  
  for some  $\tilde r$ and some $D'=\{f'(x,y)\}_{x,y\in X}$, where rows and columns of $B$ are indexed by pairs $(i,j)$ with $i,j\leq n$.  Denote by $b_{i,j}^{k,l}$ the entry at the intersection of the $(i,j)$-th column and $(k,l)$-th row of $B$. By Definition \ref{monmat},  
   $\tilde {r}(i,j)=(k,l)$ if and only if $b_{i,j}^{k,l}\neq 0$, hence $r=\tilde {r}$ (by assumption $A$ and $B$ have  non-zero entries at the same places).  
   Recall that $r(x,y)=({ }^xy, x^{y})$. Since $A$ is an $R$-matrix, the linear mapping $c'(x\otimes y)=f'(x,y){ }^xy\otimes x^{y}$ satisfies QYBE where $f'(x,y)=b_{x,y}^{r(x,y)}$.
   
   By Lemma \ref{identity},  
   for every $x,y, z\in X$  
 \[f(x,y)\cdot f(x^{y}, z)\cdot f({{ }^xy},  { }^{x^{y}}z)=f(y,z)\cdot f(x, { }^yz)\cdot f(x^{{ }^yz}, y^{z}).\]
  Similarly, 
 \[f'(x,y)\cdot f'(x^{y}, z)\cdot f'({{ }^xy},  { }^{x^{y}}z)=f'(y,z)\cdot f'(x, { }^yz)\cdot f'(x^{{ }^yz}, y^{z}).\]
   By multiplying these two equations, we get  \[g(x,y)\cdot g(x^{y}, z)\cdot g({{ }^xy},  { }^{x^{y}}z)=g(y,z)\cdot g(x, { }^yz)\cdot g(x^{{ }^yz}, y^{z})\]
  where $g(x,y)=f'(x,y)f'(x,y)=a_{(x,y)}^{r(x,y)}b_{(x,y)}^{r(x,y)}$.  
   Denote $D''=\{g(x,y)\}_{ x, y\in X}.$
   By Lemma \ref{identity}, $(X, r^{D''})$ is a braided vector space of set-theoretic type. By  Definition  \ref{monmat}, the $R$-matrix  related to 
    $(X, r^{D''})$ equals $A\circ B$, consequently the matrix  $A\circ B$ satisfies the QYBE.  
\end{proof} $\square $

\begin{proposition}\label{hadamard} Let $A$ be an $R$-matrix such that when we put all its nonzero entries to $1$ the obtained matrix  is a permutation matrix. Let $G: \mathbb C\rightarrow \mathbb C$ be a function with nonzero values such that $G(pq)=G(p)G(q)$ for all $p,q\in \mathbb C$.
 If $A$ has entries $a_{i,j}$, then the matrix  whose $i,j$-th entry is  $G(a_{i,j})$ (for each $i,j$) is an $R$-matrix.
\end{proposition}
\begin{proof}  Let the first $9$ lines be as in the proof of Proposition \ref{four}.  
  By Lemma \ref{identity}, for every $x,y, z\in X$  
 \[f(x,y)\cdot f(x^{y}, z)\cdot f({{ }^xy},  { }^{x^{y}}z)=f(y,z)\cdot f(x, { }^yz)\cdot f(x^{{ }^yz}, y^{z}).\]
    By applying function $G$ to both sides of this equation we get 
 \[G(f(x,y))\cdot G(f(x^{y}, z))\cdot G(f({{ }^xy},  { }^{x^{y}}z))=G(f(y,z))\cdot G(f(x, { }^yz))\cdot G(f(x^{{ }^yz}, y^{z})).\]
   By Lemma \ref{identity}, $(X,  r^{D'})$  is a braided vector space of set-theoretic type, where  $D'=\{G(f(x,y))\}_{x,y\in X}$. Therefore
    the matrix  whose $i,j$-th entry is  $G(a_{i,j})$ (for each $i,j$) is an $R$-matrix.
\end{proof} $\square $

The next example shows that Propositions \ref{four} and \ref{hadamard} need not hold in general for arbitrary  $R$-matrices $A$ and $B$.
\begin{example}
Let 
\[
X = \left(
\begin{array}{cccc}
1  &   0 &    0  &   0 \\
 0 &    -3  &   2  &   0 \\
  0 &    2  &   0 &    0 \\
  0  &    0 &    0    &  1  
\end{array}\right).
\]

Then $det X=-4$ and $X$ is an R-matrix. 
We can verify that the Hadamard product  $X \circ  X$ is not an R-matrix.
\end{example}
It is  a simple matter to prove the following facts using known properties of the Kronecker products.
\begin{proposition}\label{thm1}
Let   $C, D  \in \mathbb C^\mathrm{n \times n}$. Then $X=C \otimes D$  satisfies  the quantum Yang-Baxter equation  (\ref{eqs1})     if and only if 
$C^2 \otimes DCD \otimes D= C \otimes CDC \otimes  D^2$.
\end{proposition}
Notice that Proposition \ref{thm1} implies that the  Kronecker product of locally monomial matrices need not be an $R$-matrix.

The next example shows that if $X \in \mathbb C^\mathrm{n^2 \times n^2 }$ is a singular solution of  the QYBE (\ref{eqs1})  then $X^\dagger$,  called  the Moore-Penrose pseudo-inverse of $X$,  may  not satisfy (\ref{eqs1}). We recall the $X^\dagger$ is  uniquely determined by the following conditions: 
$X X^\dagger  X=X$,  $X^\dagger XX^\dagger =X^\dagger$,  $X^\dagger X=(X^\dagger X)^*$ and  $XX^\dagger=(XX^\dagger)^*$.

\begin{example}\label{pinv}
Let 
\[
C = \left(
\begin{array}{cccc}
1  &   0 &    1 &   1 \\
 0 &    1  &   1  &   2 \\
  0 &    0  &   0 &    0 \\
  0  &    0 &    0    &  0
\end{array}\right).
\]

It is easily seen at once that  $C^2=C$, so $C$ is an idempotent matrix (a projection). Now we define  $X=C \otimes I_4$.  By Propositon \ref{thm1} $X$ satisfies the QYBE (\ref{eqs1}) and we have $X^\dagger=C^\dagger  \otimes I_4$, where
\[
C^\dagger = \frac{1}{3} \,\left(
\begin{array}{cccc}
2  &   -1 &    0 &   0 \\
 -1 &    1  &   0  &   0 \\
  1 &    0  &   0 &    0 \\
  0  &    1 &    0    &  0
\end{array}\right).
\]
A simple verification show that $C^\dagger$ is not an idempotent matrix, so by Proposition \ref{thm1},  $X^\dagger$ is not a solution of the QYBE (\ref{eqs1}).  
 \end{example}

By $I_{n}$ we will denote the $n\times n$ identity matrix.
\medskip
Notice that, by proceeding analogously as on page
 296 \cite{klimyk}, we get the following observation:
\begin{proposition} \label{parameterdependent}  
Let $n$ be a natural number and let $A$ be an $n^{2}\times n^{2}$  $R$-matrix such that $A^{2}$ is the identity matrix. For each $\alpha \in \mathbb C$ the matrix  $R(x)=I_{n^{2}}+\alpha xA$ is a solution of the parameter-dependent Yang-Baxter equation  (where $x$ is the variable):
\[(R(x)\otimes  I_{n})(I_{n}\otimes R(x+y))(R(y)\otimes I_{n})=
(I_{n}\otimes R(y))(R(x+y)\otimes I_{n})(I_{n}\otimes R(x)).\]
\end{proposition}
\begin{proof} Observe that $I_{n}\otimes I_{n^{2}}=I_{n^{2}}\otimes I_{n}=I_{n^{3}}$. Note that $(\alpha x R\otimes  I_{n})=\alpha x (R\otimes I_{n})$ and 
 hence 
 \[(\alpha xR\otimes  I_{n})(I_{n}\otimes \alpha (x+y)R)(I_{n^{2}}\otimes I_{n})=
(I_{n}\otimes I_{n^{2}})(\alpha (x+y)R\otimes I_{n})(I_{n}\otimes \alpha xR).\]
Similarly,
\[(I_{n^{2}}\otimes  I_{n})(I_{n}\otimes \alpha (x+y)R)(\alpha yR\otimes I_{n})=
(I_{n}\otimes \alpha yR)(\alpha (x+y)R\otimes I_{n})(I_{n}\otimes I_{n^{2}}),\]
\[(\alpha xR\otimes  I_{n})(I_{n}\otimes I_{n^{2}})(\alpha yR\otimes I_{n})=\alpha ^{2}xyI_{n^{3}}=
(I_{n}\otimes \alpha yR)(I_{n^{2}}\otimes I_{n})(I_{n}\otimes \alpha xR).\]
 Observe also that 
$(\alpha xR\otimes  I_{n})(I_{n}\otimes I_{n^{2}})(I_{n^{2}}\otimes I_{n})+
(I_{n^{2}}\otimes  I_{n})(I_{n}\otimes I_{n^{2}})(\alpha yR\otimes I_{n})=
(I_{n}\otimes I_{n^{2}})(\alpha (x+y)R\otimes I_{n})(I_{n}\otimes I_{n^{2}})$
  and 
$(I_{n^{2}}\otimes I_{n})(I_{n}\otimes \alpha (x+y)R)(I_{n^{2}}\otimes I_{n})=
(I_{n}\otimes \alpha yR)(I_{n^{2}}\otimes I_{n})(I_{n}\otimes I_{n^{2}})+
(I_{n}\otimes  I_{n^{2}})(I_{n^{2}}\otimes I_{n})(I_{n}\otimes \alpha xR).$
 We can sum all of these equations. 
The thesis now  follows from the fact that $R$ and $I_{n^{2}}$ are $R$-matrices, and from the fact that the Kronecker product is distributive.
\end{proof} $\square $
\section{ Some methods of constructing locally monomial $R$-matrices}\label{555}

\subsection{ Orbits of set-theoretic solutions and related $R$-matrices}
 
  The following Proposition \ref{3} was inspired by  results on to the second cohomology group of racks \cite{cjks}  (see Proposition $3.8$) and on the second Yang-Baxter cohomology group of left non-degenerate cycle sets \cite{lebed} (see  Lemma 9.17). Notice however that our proposition  also holds for solutions which are not  necessarily left non-degenerate.
 
 Let notation be as in subsections \ref{123} and  \ref{braces}. By $\mathbb C$ we will denote the field of complex numbers.  
\begin{proposition}\label{3} Let $F$ be a field. Let $(X,r)$ be a set-theoretic solution of the quantum Yang-Baxter equation, $\bar {r}:V\otimes V\rightarrow V\otimes V$ be the linearisation of $r$ and let $X_{1}, \ldots , X_{m}\subseteq X$ be such that  
\begin{itemize}
\item $X=\bigcup _{i=1}^{m}X_{i}$ and $X_{i}\cap X_{j}=\emptyset $ for all $i <j\leq m$ and 
 \item  $r(X_{i}, X_{j})= (X_{j}, X_{i})$ for all $i,j\leq m$.
\end{itemize}
Let $0\neq \alpha _{i,j}\in F$  and let $r': V\otimes V\rightarrow V\otimes V$  be the linear mapping  such that 
$r'(x,y)=\alpha _{i,j} \bar {r}(x\otimes y)$ for  $x\in X_{i}$ and $y\in X_{j}$.
 Then  $r'$ satisfies the quantum Yang-Baxter equation.
\end{proposition}
\begin{proof}
 It follows from Lemma \ref{identity} (Lemma $5.7$ in \cite{andrusz}). \end{proof} $\square $

 Example \ref{o} illustrates the use of orbits as sets $X_{i}$, vis-\`a-vis Proposition \ref{3} (the definition of orbits is recalled in  Subsection \ref{1234}).
 
{\em Which subsets of braces give solutions of the  quantum Yang-Baxter equation?} Below we give some examples of such sets and construct braided vector spaces of set-theoretic type on them using Proposition \ref{3}. Some answers to this question can  also be found in \cite{bcj}.
\begin{example}\label{o} {\em Let $R$ be  a finite  left brace  and let $r'$ be the associated Yang-Baxter map (defined as in Section \ref{braces}).  
 For $b\in R$, define $Q_{b}=\{b+ab:a\in R\}.$
 Notice that for $b,c\in R$ either $Q_{b}=Q_{c}$ or $Q_{b}\cap Q_{c}=\emptyset $. 
 Let $b_{1}, b_{2}, \ldots , b_{m}\in R$ be such that  the sets $Q_{b_{1}}, Q_{b_{2}}, \ldots , Q_{b_{m}}$ are pairwise distinct.  
  For $i\leq m$ denote $X_{i}=Q_{b_{i}}$ and 
  $X=\bigcup _{i\leq m}X_{i}$. Then $r'(X_{i}, X_{j})=(X_{j}, X_{i})$ for all $i,j\leq m$ 
  and $X_{i}\cap X_{j}=0$ for $i<j\leq m$. Let $r:X\times X\rightarrow X\times X$ be a restriction of $r'$, then $(X,r)$ is a non-degenerate involutive solution of the QYBE.
  }
 \end{example}
 
\begin{example}\label{graded}{\em  Let $A$ be a finite left brace which is a left nilpotent brace (for example a nilpotent ring) and $r'$ be the associated Yang-Baxter map.
Let $m$ be a natural number such that $A^{m}=0$ and $A^{m-1}\neq 0$. For $i\leq m$ denote 
  $X_{i}=\{x\in A: x\in A^{i}, x\notin A^{i+1}\}$ and  $X=\bigcup _{i\leq m}X_{i}$. Then $r'(X_{i}, X_{j})=(X_{j}, X_{i})$ for all $i,j\leq m$ and $X_{i}\cap X_{j}=0$ for $i<j$ since $A$ is a left nilpotent brace. Let $r:X\times X\rightarrow X\times X$ be a restriction of $r',$ then $(X,r)$ is a non-degenerate involutive solution of the QYBE.} \end{example}  
 
\begin{example}\label{el}{\em Let $R$ be a finite left brace which is a left nilpotent brace (for example a nilpotent ring) and let $r'$ be the associated Yang-Baxter map. For $b\in R^{i}, b\notin R^{i+1}$ define  $Q_{b}=b+R^{i+1}$.  Notice that for $b,c\in R$ either $Q_{b}=Q_{c}$ or $Q_{b}\cap Q_{c}=\emptyset $ since $R$ is a left nilpotent brace. 
 Let $b_{1}, b_{2}, \ldots , b_{m}\in R$ be such that  the sets $Q_{b_{1}}, Q_{b_{2}}, \ldots , Q_{b_{m}}$ are pairwise distinct.  For $i\leq m$ denote $X_{i}=Q_{b_{i}}$ and $X=\bigcup _{i\leq m}X_{i}$. Then $r(X_{i}, X_{j})=(X_{j}, X_{i})$ for all $i,j\leq m$ and $X_{i}\cap X_{j}=0$ for for all $i<j\leq m$. Let $r:X\times X\rightarrow X\times X$ be a restriction of $r',$ then $(X,r)$ is a non-degenerate involutive solution of the QYBE.}
 \end{example}
 The decomposition of a left brace into the Sylow subgroups of its additive group  is important in the construction of simple braces  (see \cite{dbs, simple}). 
 A similar decomposition can  also be used to construct braided vector spaces of set-theoretic type.
\begin{example}\label{syl} 
{\em  Let $A$ be a finite left  brace and $r'$ be the associated Yang-Baxter map (defined as in Section \ref{braces}). Let $A_{1}, \ldots , A_{m}$ be the Sylow subgroups of the additive group of the brace $A$. Let $0\neq \alpha _{i,j}\in \mathbb C$ for $i,j\leq m$.  For each $i\leq m$  let $X_{i}\subseteq A_{i}$ contain all elements of $A_{i}$ except the zero element. Denote $X=\bigcup _{i\leq m}X_{i}$.
 Then $r'(X_{i}, X_{j})=(X_{j}, X_{i})$ for all $i,j\leq m$ and $X_{i}\cap X_{j}=0$ for $i<j$ since $A$ is a left nilpotent brace. Let $r:X\times X\rightarrow X\times X$ be a restriction of $r',$ then $(X,r)$ is a non-degenerate involutive solution of the QYBE.}
\end{example} 
 We get the following corollary from Proposition \ref{3}:
\begin{corollary}
 Let $(X,r)$  and $X_{i}$ be as in Example \ref{o}, Example \ref{graded}, Example \ref{el} or Example \ref{syl}.
  Let $0\neq \alpha _{i,j}\in \mathbb C$ for all $i,j\leq m$. Let  $D=\{d_{x,y}\}_{x, y\in X}$ where $d_{x,y}=\alpha _{i,j}$ for $x\in X_{i}$, $y\in X_{j}$, then $(X, r^{D})$ is a braided vector space of set-theoretic type. 
  If  $\alpha _{i,j}\bar {\alpha} _{i,j}=1$ for all $i,j\leq m$ then the $R$-matrix associated to  $(X, r^{D})$ (as in Definition \ref{monmat}) is a unitary matrix.
\end{corollary}
 
 \subsection{Some other methods of constructing braided vector spaces of set-theoretic type}\label{6}
 
Notice that a braided vector spaces of set-theoretic type (BVST) can be obtained by a diagonal similarity (as in
the case of trivial  BVST in Section \ref{123}) and by the decomposition of a set-theoretic
solution into invariant subsets as in  Proposition \ref{3}. 
In this section  we give some examples of braided vector spaces of set-theoretic type which cannot be obtained by combining these two methods. All examples in this subsection satisfy the assumptions 
 of Lemma \ref{inna} and  therefore can be used to construct extensions of set-theoretic solutions as in Proposition $2.9$ 
  in \cite{car}.
 \begin{example}\label{hura5}  
{\em  Let $X=\{[i, j]: i, j\in {\mathbb {Z}\over 2\cdot \mathbb {Z}} \}$, and 
 define $r:X\otimes X\rightarrow X\otimes X$ as 
\[r([i,j], [m,n])=([m+1, n+m+i], [i+1, j+i+m]).\]
 Let $g: {\mathbb {Z}\over 2\cdot \mathbb {Z}} \times {\mathbb {Z}\over 2\cdot \mathbb {Z}} \rightarrow \mathbb C$  be an  arbitrary function with non-zero values such that $g(1,0)=g(0,1)$.    
Define $f([i,j], [m,n])=g( i+j+n, m+j+n).$  
 Let $D=\{d_{x,y}\}_{x,y\in X}$ be such that $d_{x,y}=f(x,y)$ for all $x,y\in X$; then  $(X, r^{D})$ is a braided vector space of set-theoretic type. Since $|X|=4$, this example yields a $16\times 16$ $R$-matrix, given in  Example \ref{A2}, which  can  be verified by hand. }
\end{example}
 
\begin{example}\label{simple}
 {\em Let $X={\mathbb {Z}\over n\cdot \mathbb {Z}}$ and $r(x,y)=(y+1, x-1),$ for $x,y\in X$. Observe that $(X,r)$ is a non-degenerate, involutive, indecomposable set-theoretic solution of the QYBE. 
 Let $g:{\mathbb {Z}\over n\cdot \mathbb {Z}}\rightarrow \mathbb C$ be any mapping with nonzero values, and define \[d_{i,j}=f(i,j)=g(i-j)\] for $i,j\in {\mathbb {Z}\over n\cdot \mathbb {Z}}$. We will show that $(X, r^{D})$ is a braided vectors space of set-theoretic type, where $D=\{d_{i,j}\}_{i,j\in  {\mathbb {Z}\over n\cdot \mathbb {Z}}}$. 
  Notice that $f$ satisfies  the assumptions of Lemma \ref{inna}, since 
  $f(x,y)= g(x-y)=f(x-1, y-1)=f(x^{{ }^yz}, y^{z})$,   $f(x^{y}, z)=f(x-1,z)=g(x-1-z)=f(x, z+1)= 
f(x, { }^yz)$, 
$f({{ }^xy},  { }^{x^{y}}z)=f(y+1, z+1)=g(y-z)=f(y,z).$
 
We can assume that $g(1)^{2}\neq g(0)g(0)$, 
 then $f(1,0)f(1,0)\neq f(0,0)f(1, 1)$, and by Lemma \ref{trivial}  $(X, r^{D})$ is a   a braided vector space of set-theoretic type, moreover $(X,r^{D})$ is  non-trivial. This example is illustrated by Examples \ref{A1} and \ref{A3}.}
\end{example}
 Example \ref{simple} yields the following corollary.
\begin{corollary}\label{niunius}
 For every natural number $n>1$ there exists a non-degenerate, involutive,  indecomposable solution $(X,r)$ of the QYBE of cardinality $n$ with a non-trivial braided vectors space of set-theoretic type $(X, r^{D})$, for some $D$. 
\end{corollary}
\subsection{Unitary $R$-matrices obtained by $I$-retraction}

  Recall that a subset $I$ of left brace $A$  is an ideal if $x+y\in I$ and $z\cdot x, z\cdot x\in I$ for all $x, y\in I$ and $z\in A$, see \cite{rump}. For $x\in X$ let $[x]_{I}=\{x+i: i\in I\}$.
 
 On page 160 in \cite{rump}, Rump  showed that if $I$ is an ideal in a brace $R$ then the factor brace $R\over I$ is well defined.  
 We will use this fact  to introduce the $I$-retraction; 
  the $I$-retraction can be viewed as a generalisation of the retraction technique. We will use the $I$-retraction for construction of locally monomial BVS.
\begin{proposition}\label{ideal}[ $I$-retraction ]
Let $A$ be a left brace and  let $(X,r)$ be a set-theoretic solution of the QYBE such that $X\subseteq A$ and $r$ is a restriction of the Yang-Baxter map associated to $A$. 
  Let $I$ be an ideal in $A$, and denote $[x]_{I}=\{x+i: i\in I\}$ for $x\in X$. Define the 
 relation $\sim $ on sets $[x]_{I}$ by saying that $[x]_{I} \sim  [y]_{I}$ if and only if $x-y\in I$.  Then $\sim $ is an equivalence relation. Let $X_{I}$ be the set of equivalence classes of this relation. Denote 
$r(x,y)=({ }^xy, x^{y})$ for $x,y\in X$. Define 
 $\tilde {r}([x]_{I},[y]_{I})=([{ }^xy]_{I}, [x^{y}]_{I})$.
 Then $(X_{I}, \tilde {r})$ is a non-degenerate involutive set-theoretic solution of the QYBE. 
\end{proposition}
\begin{proof}  This follows from the result that the factor brace $A/I$ is well defined, which was proved by Rump (page 160, end of section $2$ in \cite{rump}).
\end{proof} $\square $

Recall that factors of $A/I$ are well defined for any skew brace $A$ and any ideal $I$ in $A$, see \cite{gv}, therefore  Remark \ref{ideal} can be generalised for non-degenerate solutions which can be embedded in skew braces.

\begin{lemma}\label{retr2}
 Let $(X,r)$ be a solution of a finite multipermutation level. If $([X]_{I}, r_{I})$ is an $I$-retraction of solution $(X,r)$,  then $([X]_{I}, r_{I})$ is a solution of a finite multipermutation level and $mpl ([X]_{I}, r_{I})\leq mpl(X,r)$.   
\end{lemma}
 \begin{proof} It follows from Proposition $4.7$ \cite{tatiana}, since if the property from Proposition $4.7$  \cite{tatiana} holds for all $a,b, y_{1}, \ldots , y_{m}\in X$ then it  holds for all  $[a]_{I}, [b]_{I}, [y_{1}]_{I}, $ $\ldots , [y_{m}]_I\in [X]_{I}.$    
\end{proof} $\square $

 The proof of  the following proposition follows from Lemma $5.7$ in \cite{andrusz} (see Lemma \ref{identity} in Section \ref{6}).
\begin{proposition}\label{last} Let $A$ be a left brace, $I$ be an ideal in $A$ and $(X,r)$, $(X_{I}, \tilde {r})$ be as in  
  Proposition \ref{ideal}.  Let $ {D}=\{d_{[x]_{I}, [y]_{I}}\}_{[x]_{I}, [y]_{I}\in X_{I}}$  for some $d_{[x]_{I}, [y]_{I}}\in \mathbb C$.  Define ${D}'=\{d'_{x,y}\}_{x,y\in X}$, where $d'_{x,y}=d_{[x]_{I}, [y]_{I}}$ for all $x,y\in X$. 
 If $(X_{I}, \tilde {r}^{D})$ is a braided vector space of set-theoretic type then  $(X, r^{{D}'})$ is a braided vector space of set-theoretic type.
\end{proposition}
 Notice that Proposition \ref{last} also holds  if instead of $I$-retraction we consider the classical retraction (for some related results on extensions and retraction  of cycle sets see Proposition 10 \cite{lebed}).

\begin{example}\label{ring} {\em Let $n\geq 2$ be a natural number. 
  Let $F$ be the  two-elements field and let $A$ be the free (non-unital)  $F$-algebra generated by one generator $x$ subject to relation $x^{n}$, then $A$ is a nilpotent algebra.    Let $X=\{x+xf: f\in A\}$, and let $r$ be a restriction of the Yang-Baxter map associated to $A$. 
  Let $g: {\mathbb {Z}\over 2\cdot \mathbb {Z}} \times {\mathbb {Z}\over 2\cdot \mathbb {Z}} \rightarrow \mathbb C$  be an  arbitrary function with non-zero values such that $g(1,0)=g(0,1)$.    

Let $g,h\in A, i,j,m,n\in {\mathbb {Z}\over 2\cdot \mathbb {Z}}$. For $u=x+ ix^{2}+jx^{3}+gx^{3}$ and $v=x+ mx^{2}+nx^{3}+hx^{3}$
 define $f(u, v)=g( i+j+n, m+j+n).$  
 Define $D=\{d_{u,v}\}_{u,v\in X}$ as $d_{u,v}=f(u,v)$, then  
$(R, r^{D})$ is a braided vector space of set-theoretic type. 
 It follows from  Proposition \ref{last}  because Example \ref{hura5} is  the $I$-retraction of Example \ref{ring} for  $I= A^{4}$. }
\end{example}
\begin{example}\label{ring2}
{\em Let $X, A, r, n$ be as in Example \ref{ring}.  
   Let $g:{\mathbb {Z}\over 2\cdot \mathbb {Z}}\rightarrow \mathbb C$ be any mapping with nonzero values.
 Let $g,h\in A, i,j\in {\mathbb {Z}\over 2\cdot \mathbb {Z}}$. For $u=x+ ix^{2}+x^{2}g$ and $v=x+ jx^{2}+hx^{2}$
   define $f(u, v)=g(i-j).$      
 Define $D=\{d_{u,v}\}_{u,v\in X}$ as $d_{u,v}=f(u,v)$, then  
$(R, r^{D})$ is a braided vectors space of set-theoretic type. 
 It follows from  Proposition \ref{last}  because Example \ref{simple}, with $p=2$, is  the $I$-retraction of Example \ref{ring2} for  $I= A^{3}$. }
\end{example}
\section{Indecomposable solutions}\label{ind}

 In \cite{etingof}, Etingof, Schedler and Soloviev showed  that each  involutive non-degenerate indecomposable solution of cardinality $p$, for each prime number  $p$, is isomorphic to the permutation solution $(X,r)$ where  $X={\mathbb {Z}\over p\cdot \mathbb {Z}}$ and $r(x,y)=(y-1, x+1)$. 
  In \cite{guralnick} Etingof, Guralnick and Soloviev showed that all non-degenerate indecomposable solutions of the QYBE whose cardinality is a prime number are affine.
  
   The purpose of this section is to show that for other cardinalities the situation is more complicated. 
  We will introduce one-generator braces to show that every finite one-generator brace yields an indecomposable solution of the QYBE:
  \begin{definition}\label{onegenerator}[one-generator brace]
 Let $A$ be a left  brace and $x\in A$, and by $A(x)$ we will denote the smallest left  brace which contains $x$ (under the same operations  of $+$ and $\circ $  as in $A$). If $A=A(x)$ for some $x\in A$ then we say that $A$ is a left  brace generated by one element $x$, or a {\em one-generator left brace}.
 Notice that if $A$ is a finite brace, then every element from $A(x)$ can be obtained by applying several times operations $+$ and $\circ $ to element $x$; since inverses of elements in groups $(G, +)$ and $(G, \circ )$ are powers of these elements since these groups are finite.  
  \end{definition}
   
   Notice that one-generator braces which are Jacobson radical rings are commutative. The following examples show that one-generator braces are abundant and need not be not commutative:  
  
  \begin{example}
   Let $(S, \bullet , + )$ be  the brace constructed  in Proposition $2.24$ in \cite{sv},  then any element in  $S$ generates a finite one-generator brace.  Moreover, the obtained  one-generator braces have a finite multipermutation level and usually are not commutative. 
     \end{example}
   Recall that a group $A$ factorises  through two subgroups $B, C$ is $A=BC=\{bc: b\in B, c\in C\}$. The factorisation is exact if $B\cap C=1$. The following example is inspired by Theorem $2.3$  and Proposition  $2.24$ from \cite{sv} but has a different additive group. This example is useful for constructing examples of one-generator braces. 
   Recall that for a ring $(R, +, \cdot )$ operation $\circ $ is defined as $a\circ b=a+b+a\cdot b$ for $a,b\in R$. 
  \begin{proposition} Let $N$ be a nilpotent ring whose group $(N, \circ )$
 admits  an exact factorisation through two subgroups $B$ and $C$, so $N=B\circ C$.
    Then  $N$ with the binary operations $+$ and $\odot $ is a brace where 
    $+$ is the usual addition in the ring $N$  and $\odot $ is defined for $a, a'\in N$ as 
$a\odot a' =b\circ a'\circ c$, where $a=b\circ c$ with $b\in B$, $c\in C$.
  \end{proposition}
  \begin{proof}
    It was shown in Theorem $2.3$ \cite{sv} that $(N, \odot)$ is a group. 
     Observe that for $a, e, f\in N$ with $a=b\circ c$  with $b\in B$, $c\in C$ we have
     $ a\odot (e+f)+a=b\circ (e+f)\circ c+a=2b+2c+2bc+e+f+be+bf+ec+fc+bec+bfc=
     b\circ (e)\circ c+b\circ (f)\circ c =a\odot e+a\odot f.$
    Therefore $(N, +, \odot)$ is a brace.
\end{proof} $\square $
\begin{theorem}\label{567} 
 Let $(X,r)$ be a finite  non-degenerate, involutive solution of the QYBE. The following are equivalent:
\begin{itemize}
\item[1.] $(X,r)$ is an indecomposable solution.
\item[2.]  There exist  a brace $A$ and $x\in A$ such that
$X=\{x+ax:a\in A\}$ and every element of $A$ is a sum of elements from $X$. 
Moreover $r$ is a restriction of the Yang-Baxter map associated to $A$.
\end{itemize}
\end{theorem}
\begin{proof} $1\rightarrow 2$. Let $(X,r)$ be an indecomposable solution. By Proposition \ref{1256} there exists a left brace $A$ such that $X\subseteq A$ and $r$ is a restriction of the Yang-Baxter map associated to $A$. Moreover every element of $A$ is a sum of some elements from $X$, and $X$ generates the group $(A, \circ )$. 
 Observe that since the brace $A$  is finite then the inverse of an element $x$ in a group $(A, \circ )$ equals some power of this element in this group. 
   Fix $x\in X$, and let $y\in X$. Since $(X,r)$ is indecomposable, there are $r_{1}, \ldots ,r_{n}\in X$ such that $\sigma _{r_{1}}\sigma _{r_{2}}\ldots \sigma _{r_{n}}(x)=y$, so
$\sigma _{r}(x)=y$ where $r=r_{1}\circ \cdots \circ r_{n}$. Recall that $\sigma _{r}(x)=x+rx$.  It follows that $X\subseteq  \{ x+rx: r\in A\}$.

 It remains to show that $\{x+rx:r\in A\}\subseteq X$.
  Recall that $X$ generates the group $(A, \circ )$. 
Observe that since the brace $A$  is finite then the inverse of an element $x$ in a group $(A, \circ )$ equals some power of this element in this group. Therefore every element of $A$ can be written as $r=r_{1}\circ \cdots \circ r_{n}$ where $r_{i}\in X$, and since $\sigma _{y}(x)\in X$ for $x, y\in X$ we get $\sigma _{r}(x)=\sigma _{r_{1}}\sigma _{r_{2}}\ldots \sigma _{r_{n}}(x)\in X$, hence  $\{ x+rx: r\in A\}\subseteq X$.

 $2\rightarrow 1$. We will first show that $r(X,X)\subseteq (X,X)$. Notice that for each $r, r_{1}\in A$ we have  $\sigma _{r_{1}}(x+rx)=\sigma _{r_{1}\circ r}(x )=x+(r_{1}\circ r)x\in X$. Let $a^{-1}$ denote the inverse of $a$ in the group $(A, \circ )$. By the definition  $\tau _{y}(z)=\sigma _{\sigma _{z}(y)^{-1}}(z)$ for all $y, z\in A$, hence $\tau _{y}(z)\in X$.  Therefore $r(X,X)\subseteq  (X,X)$.
We will show that every element $r\in A$ can be written as $r=r_{1}\circ \cdots \circ r_{n}$ for some $r_{1}, \ldots , r_{n}\in X$. By assumption $r=\sum_{i=1}^{n}x_{i}$ for some $n$ and some $x_{i}\in X$, where  $x_{i}=s_{i}x+x$ for some $s_{i}\in A$.  By Theorem $3.8$ \cite{tatiana} $a+b=a\circ \sigma _{a^{-1}}(b)$ for $a,b\in A$. We will proceed by induction on $n$, and assume that  if $r=\sum_{i=1}^{n}x_{i}$ then $r=y_{1}\circ \cdots \circ y_{m}$ for some $y_{i}\in X$; let  
$r'=\sum_{i=1}^{n+1}x_{i}=r+x_{n+1}=r\circ \sigma _{r^{-1}}(x_{n+1})=y_{1}\circ \cdots \circ y_{m+1}$ where $y_{m+1}=\sigma _{r^{-1}}(x_{n+1})=\sigma _{r_{1}}\cdots \sigma _{r_{1}}(x_{n+1})\in X$ by the first part of this proof.
 We have shown that if $r\in A$ then $r=r_{1}\circ \cdots \circ r_{n}$ for some $r_{1}, \ldots , r_{n}\in X$.
 Therefore $x+rx=\sigma _{r}(x)= \sigma _{r_{1}}\sigma _{r_{2}}\ldots \sigma _{r_{n}}(x)$ where $r_{i}\in X$, hence 
 $(X,r)$ is indecomposable.
\end{proof} $\square $
 
\begin{theorem}\label{onegenerator2}
 Let $A$ be a  finite left  brace, let $x\in A$ and let $A(x)$ be as in Definition \ref{onegenerator}. 
 Denote $X=\{x+ax: a\in A(x)\}$, 
  and let $r:X\times X\rightarrow X\times X$ be a restriction of the Yang-Baxter map associated to $A$.
  Then $(X,r)$  is an indecomposable, non-degenerate, involutive  solution of the  quantum Yang-Baxter equation. 
\end{theorem}
\begin{proof}  Notice that since $A$ is finite then $-x$ is a sum of 
$n$ copies of $x$ for some $x$. By Theorem \ref{567}  it suffices to show that every element in $A(x)$ is a sum of elements from the set $\{x, ax:a\in A\}$. Let $c\in A(x)$, since every element of $A(x)$ can be obtained by applying operations $+$ and $\cdot $ to $x$, and the number of elements is finite, then $A(x)=\bigcup _{i=1}^{n} R_{i}$ for some $n$ where $R_{1}=\{x\}$ and inductively $R_{j+1}=\{ r=\sum _{k}r_{k}:r_{k}\in \bigcup _{i=1}^{j} R_{i}\}\cup \{ r=r_{1} r_{2}:r_{1}, r_{2}\in \bigcup _{i=1}^{j} R_{i}\}$. 

Let $y\in A(x)$; we will show that 
$y=k\cdot x+\sum_{i}s_{i}x$ for some natural number $k$ and for some $s_{i}\in A(x)$, where $k\cdot x$ denotes the sum of $k$ copies of $x$. We will proceed by induction on $n$ where $y\in R_{n}$. If $n=1$ then $y=x$ so the result holds. Suppose the result folds for all $y\in R_{i}$ for $i\leq n$ and let $y\in R_{n+1}$. Then by the definition either $r=\sum _{k}r_{k}:r_{k}\in \bigcup _{i=1}^{n} R_{i}$ and the result holds by the inductive assumption or $r=r_{1}r_{2}$ for some $r_{1}, r_{2}\in \bigcup _{i=1}^{n} R_{i}$. By the inductive assumption $r_{2}=k'\cdot x+\sum_{i=1}^{m}s_{i}x$ for some $s_{i}\in A(x)$ and some natural number $k'$. Hence $r=k'\cdot (r_{1}x)+\sum_{i=1}^{m} r_{1}\cdot (s_{i}x)$. Notice that  $r_{1}(s_{i}x)=(r_{1}+s_{i}+r_{1}s_{i})x-r_{1}x-s_{i}x$,  and since $A(x)$ is finite $-r_{1}x=\sum_{i=1}^{k}r_{1}x$ and $-s_{i}x=\sum_{j=1}^{k'}s_{i}x$ for some $k,k'$. Therefore  $r_{1}(s_{i}x)$ is a sum of elements from the set $\{ x, ax:a\in A\}$,  concluding the proof. 
\end{proof} $\square $
\begin{question}
 Characterise one-generator braces of the multipermutation level $2$.
\end{question}
\begin{question} Is Theorem \ref{onegenerator2} also true for infinite one-generator braces? 
 \end{question}

\section{ Nilpotent braces and related solutions} \label{nil2}
In this this section  
 a non-degenerate, involutive set-theoretic solution $(X,r)$ of the  quantum Yang-Baxter equation will simply be referred to as {\em a solution}.
 We will investigate solutions which can be embedded into  nilpotent braces.
 It was shown in  Proposition $5.15$ \cite{tatiana} that  if $G(X,r)$ is the structure group of $(X,r)$, then $mpl (G,r)\leq mpl(X,r)+1$, where $mpl(G,r)$ denotes the multipermutation level of solution associated to $G(X,r)$.
 We obtain a similar result.
\begin{lemma}\label{multip} 
 Let $(X,r)$ be a finite solution and let $A$ be a finite left brace such that $X\subseteq A$ and $r$ is a restriction of the Yang-Baxter map associated to $A$. 
 Assume that every element of $A$ is a sum of some elements from $X$.
 If $(X,r)$ has a finite multipermutation level,  then $A$ is a right nilpotent brace and  $A^{(m+2)}=0$ where $m=mpl (X,r)$.
\end{lemma}
 \begin{proof} Let $[x]$ denote the retraction of element $x\in X$. Notice that  by the definition of the retraction for $x,y\in X$ we have $[x]=[y]$ if and only if $x\cdot r=y\cdot r$ for all $r\in X$. Since every element of $A$ 
  is a sum of elements from $X$, this is equivalent to say that $x\cdot s=y\cdot s$ for all $s\in A$.  Therefore the retraction of the solution $(X,r)$ embeds into the retraction of the solution associated to brace $A$. By Proposition $7$ \cite{rump} this implies that 
 the retraction of $(X,r)$ embeds into brace $A/Soc (A)$. Notice also that every element in $A/Soc (A)$ is a sum of elements from the set [X]-the retraction of $X$.  
 Therefore, the fact that $A^{(m+2)}=0$ can be proved by induction on $m$ (by Proposition $6$ \cite{CGIS}). 
\end{proof} $\square $

\begin{lemma} \label{1277}
 Let $(X,r)$ be a solution  which is  the retraction of a finite solution whose  permutation group  $(X,r)$ is nilpotent.  Then there exists a finite  left brace $A$ such that 
 $X\subseteq A$ and every element of $A$ is a sum of elements from $X$ and $A$ is a left nilpotent brace, where $r$ is a restriction of the Yang-Baxter map associated to $A$. 
\end{lemma}
\begin{proof} 
 Let  $(X, r)$ be the retraction of solution $(Y,r')$ and denote  $r'(x,y)=(\sigma _{x}(y), \tau _{y}(x)).$
 By Theorem \ref{1255} and Corollary \ref{1256},  there is a finite left brace $B$ such that $Y\subseteq B$ 
  and every element of $B$ is a sum of elements from $Y$, and elements from $Y$  generate the  multiplicative group $(B, \circ )$; moreover $r'$ is the restriction of  $r'':B\times B\rightarrow B\times B$ - the Yang-Baxter map associated to $B$. Denote $r''(x,y)=(\sigma '_{x}(y), \tau '_{y}(x)).$ By assumption $r'$ is the restriction of $r''$, so $\sigma '_{y}(x)=\sigma _{y}(x)$ for $x,y\in Y$.
 
   Let $\mathcal {G}(Y,r')$ be the permutation  group of $(Y,r')$, so it is the group generated by maps 
  $\sigma _{x}:Y\rightarrow Y$ for $x\in Y$ with the operation of the composition of maps. 
 Let $T$ be the group generated by maps $\sigma '_{x}:B\rightarrow B$ for $x\in Y$.
   Recall that elements from $Y$  generate the multiplicative group $(B, \circ )$ and $\sigma _{x}\sigma _{y}=\sigma _{x\circ y}$, therefore 
   $T$ equals the group generated by maps $\sigma '_{b}:B\rightarrow B$ for $b\in b$.

 Let $f\in T$, then $f$  is the identity map if and only if $f(b)=b$ for all $b\in B$.
  This is equivalent to saying that $f(y)=y$ for all $y\in Y$, since every element of $B$ is a sum of elements from $Y$ and $\sigma _{y}(a+b)=\sigma (a)+\sigma (b)$.  Therefore, 
  $T$ is isomorphic to the group $\mathcal {G}(Y,r')$, which implies that $T$ is nilpotent (since $T$ and $\mathcal {G}(Y,r')$ have corresponding generators and $\mathcal {G}(Y,r')$ is nilpotent).

Notice that  by the definition of the retraction for $x,y\in Y$ we have $[x]=[y]$ if and only if $x\cdot r=y\cdot r$ for all $r\in Y.$  Since every element of $B$ 
  is a sum of elements from $Y$, this is equivalent to saying that $x\cdot s=y\cdot s$ for all $s\in B$.  Therefore the retraction of the solution $(Y,r')$ embeds into the retraction of the solution associated to the left  brace $B$. By Proposition $7$ \cite{rump} this implies that 
 the retraction of $(X,r)$ embeds into brace $A=B/Soc (B)$. Notice also that every element in $A$ is a sum of elements from [X]-the retraction of $X$.

 It remains to show that the multiplicative group $(A, \circ )$ of $A$ is nilpotent. 
  Recall that $T$ is nilpotent, therefore there is $n$ such that for every $b_{1}, \ldots , b_{n}\in B$ 
   $\sigma '_{[[\ldots [[b_{1}, b_{2}] b_{3}]\ldots ]b_{n}]} $ is the identity map, therefore 
   $\sigma '_{[[\ldots [[b_{1}, b_{2}] b_{3}]\ldots ]b_{n}]}(b)=b$ for all $b\in B$, consequently 
   $[[\ldots [[b_{1}, b_{2}] b_{3}]\ldots ]b_{n}]$ is in the socle of $B$. 
 It follows that $A^{\circ }$ is a nilpotent group (since $A=B/Soc(B)$). By Theorem $1$ from \cite{smok} $A$ is a left nilpotent brace, so $A^{n}=0$ for some $n$.
\end{proof} $\square $

 If $(X,r)$ is a  solution of a finite multipermutation level and the permutation group of $(X,r)$ is nilpotent then $(X,r)$ need not to embed in a left nilpotent left  brace; for example a right nilpotent  left brace of cardinality $6$ whose multiplicative group is not nilpotent and whose permutation group is nilpotent  can be obtained by taking the opposite multiplication in Example $3$ \cite{rump}.

\begin{lemma}\label{fajny}
 Let $s$ be a natural number and let $A$ be a left brace which is right nilpotent, so $A^{(s)}=0$ for some $s$.
 Let $a, b\in A$.  
Define inductively elements $d_{i}=d_{i}(a,b), d_{i}'=d_{i}'(a, b)$  as follows:
$d_{0}=a$, $d_{0}'=b$, and for $i\leq 1$ define $d_{i+1}=d_{i}+d_{i}'$ and $d_{i+1}'=d'_{i}d_{i}$.
 Then for every $c\in A$ we have
\[(a+b)c=ac+bc+\sum _{i=0}^{2s} (-1)^{i+1}((d_{i}'d_{i})c-d_{i}'(d_{i}c)).\]
\end{lemma}
\begin{proof} We can prove by induction that $d_{i}'\in A^{(i+1)}$, hence almost all $d_{i}'$ are zero. We can use the same proof as in Lemma $15$ in \cite{smok} where instead of $d_{j}d_{j}'$ we write at each place $d_{j}'d_{j}$ for every $j$ (for various $j$), and instead of $ab$ we write  $ba$.
\end{proof} $\square $

\begin{lemma}\label{retr}
 Let $(X,r)$ be a solution of a finite multipermutation level  and $Y\subseteq X$ be such that $r(Y,Y)\subseteq (Y,Y)$, then $(Y,r')$ is also a solution of a finite multipermutation level and 
 $mpl (Y,r')\leq mpl (X,r)$, where $r'$  is a restriction of $r$ to $Y\times Y$.   
\end{lemma}
 \begin{proof} This follows from Proposition $4.7$ \cite{tatiana}, since if the property from Proposition $4.7$  \cite{tatiana} holds for all $a,b, y_{1}, \ldots , y_{m}\in X$ then it holds for all 
$a,b, y_{1}, \ldots , y_{m}\in Y$.   
\end{proof} $\square $

 The first part of Lemma \ref{retr} was proved in \cite{cjo} under the additional assumption that  $(Y, r')$  is invariant under the permutation group  of $(X,r)$.

\begin{theorem}\label{multipermutation}
 Let $(X,r)$ be a finite solution.  The following are equivalent:
\begin{itemize}
\item[1.] $(X,r)$ is an indecomposable solution of a finite multipermutation level, and  $x$ is an element of $ X$.
\item[2.]  There is a finite one-generator left brace $A$ generated by some element $x\in A$ such that $X=\{x+ax:a\in A\}$ and $r$ is a restriction of the Yang-Baxter map associated to $A$. Moreover,  $A^{(m)}=0$ for some $m$. 
\end{itemize}
\end{theorem}
  \begin{proof} If point $2$ holds then by Theorem \ref{onegenerator2} $(X,r)$ is an indecomposable solution. By Lemma \ref{retr} $(X,r)$ has a finite multipermutation level, as it is a subsolution of the solution associated to a right nilpotent left brace.
Recall that by Proposition $5$ from \cite{CGIS}  the solution associated to a right nilpotent brace has a finite multipermutation level.
 This shows implication $2\rightarrow 1$. 

 Assume that point $1$ holds, and let $A$ be as in  Theorem \ref{1255} and Corollary \ref{1256}. By Lemma \ref{multip}, $A$ has a finite multipermutation level.
 Let $x\in X$, it suffices to show that $A=A(x)$-the brace generated by $x$. Recall that $A^{(n)}$ is an ideal in $A$, by a result of Rump \cite{rump}.
 We will show by induction that $A\subseteq A(x)+A^{(n)}$ for every $n$. 
 For $n=1$ it is true as $A^{(1)}=A$.
 Suppose the result holds for some $n$, and let $y\in A$, since  every element of $A$ is a sum of elements from $X$, 
 then $y$ is a sum of some elements from $X$ so 
$y=k\cdot x+\sum_{i}r_{i}x$ for some natural number $k$ and for some $r_{i}\in A$, where $k\cdot x$ denotes the sum of $k$ copies of $x$. By assumption $r_{i}\in A(x)+A^{(n)}$ hence $r_{i}=x_{i}+s_{i}$, where $x_{i}\in A(x)$, $s_{i}\in A^{(n)}$. 
 Notice that $r_{i}x=(x_{i}+s_{i})x\in x_{i}x+A^{(n+1)}\subseteq A(x)+A^{(n+1)}$ by Lemma \ref{fajny} applied for $a=x_{i}$ and $b=s_{i}$. 
 It follows that $y\in  A(x)+A^{(n+1)}$. By the assumptions, $A^{(m)}=0$ for some $m$, therefore $A\subseteq A(x)$, as required. 
\end{proof} $\square $

A solution $(X,r)$ is square-free if $r(x,x)=(x,x)$ for all $x\in X$. It was shown by Rump in \cite{rump2} that every finite square-free solution of the QYBE is decomposable.
\begin{corollary}\label{bi}
 Let $(X,r)$ be a finite solution of a finite multipermutation level. If for every $x\in X$ there is $y\in X$ such that $r(x,y)=(y,x)$, then the solution $(X,r)$ is decomposable, provided that the cardinality of $X$ is larger than $1$. 
\end{corollary}
\begin{proof}
 Let $x\in X$, then by Theorem \ref{multipermutation} $X=\{x+ax:a\in A\}$ where $A$ is a left brace generated by $x$. Let $r(x,y)=(y,x)$, since $y\in X$ then $y=x+rx$ for some $r\in A$. Notice that $r(y,x)=(x,y)$ since $r$ is involutive, this implies $yx+x=x$, hence $0=yx=(x+rx)x=x^{2}+a$ for some $a\in A^{(3)}$ by applying Lemma  \ref{fajny}  for $a=x$ and $b=rx$. Therefore, $x^{2}\in A^{(3)}$ which implies $A^{(2)}\subseteq A^{(3)}\subseteq A^{(4)}\cdots \subseteq A^{(n)}=0$ for some $n$, hence $x^{2}=0$, and so the cardinality of $X$ is $1$.
\end{proof} $\square $
\begin{question}
 Are Theorem \ref{multipermutation} and  Corollary \ref{bi} also true without the assumption that the  multipermutation level of $(X,r)$ is finite? 
\end{question}

\begin{question}
 Is there a connection between one-generator skew-braces and indecomposable non-degenerate
 set-theoretic solutions of the QYBE?
  \end{question}

\section{ $R$-matrices constructed using indecomposable solutions and nilpotent braces}\label{33333}
 In this section we will use results from Sections \ref{ind} and  \ref{nil2} to construct $R$-matrices related to indecomposable  set-theoretic solutions of the QYBE.
\begin{proposition}\label{stronglynilpotent} Let $X$ be a set which has more than one element.
 Let $ (X,r)$ be a non-degenerate, involutive, indecomposable set-theoretic solution of the QYBE of a finite multipermutation level. Suppose that $(X,r)$ is the  retraction of a finite solution whose permutation group is nilpotent.  Then there exists  an $I$-retraction of $(X,r)$ such that  $(X_{I}, r_{I})$ is a solution 
 isomorphic to the solution $(Y, \tilde {r})$ where $Y={\mathbb {Z}\over n\cdot \mathbb {Z}}$ 
 for some $n>1$  and $\tilde {r}(x,y)=(y+1, x-1),$ for $x,y\in Y$.
\end{proposition}
\begin{proof} By Lemma \ref{1277}, there is a left nilpotent brace $A$ such that $X\subseteq A$ and every element of $A$ is a sum of elements from $X$, and $r$ is a restriction of the Yang-Baxter map associated to $A$. We can use this particular brace $A$ in the proof of implication $1\rightarrow 2$  of Theorem \ref{multipermutation} and get   that $X=\{ax+x:a\in A\}$ for some $x\in A$, $A$ is generated as a brace by $x$ and $A^{(m)}=0$
 for some $m$. Since $A$ is both a left nilpotent and right nilpotent left brace, by Theorem $3$ from \cite{smok} we have  $A^{[t]}=0$ for some $t$, where  $A^{[1]}=A$ and $A^{[n+1]}$ is defined inductively as  $A^{[n+1]}=\sum_{i=1}^{n}A^{[i]}\cdot A^{[n+1-i]}$.
      
 Notice that  $A^{2}=0$ implies $r(x,y)=(y,x)$;  then $(X,r)$ is decomposable, since each 
 element in $X$ is an orbit, hence $A^{2}\neq 0$.
 Let $I=A\cdot A^{2}+A^{2}\cdot A$, then it is easy to see that $I$ is an ideal in $A$. Notice that $A/I$ is a ring, 
 since $(a+b)c=((a+b)+ab+ab(a+b))c=(a+b+ab)c=ac+bc$ for all $a,b,c\in A/I$.
 Let $\bar {x}=x+I\in A/I$, then $X_{I}=\{\bar {x}+r\bar {x}: r\in A/I\}$ and $A/I$ is a ring generated by element $\bar {x}$ where $\bar {x}^{3}=0$. Hence every element of $X_{I}$ equals $\bar {x}+k\cdot \bar {x}^{2}$ for some $k$ (where $k\cdot \bar {x}$ is a sum of $k$ copies of $\bar {x}$). Notice that ${\bar x}^{2}\neq 0$, since otherwise $A^{2}\subseteq I=A^{[3]}$. Notice that  $A^{2}\subseteq A^{[3]}$ yields $A^{2}\subseteq A^{[3]}\subseteq A^{[4]}\subseteq \ldots =0$  since $A^{[t]}=0$. Therefore $A^{2}=0$, a contradiction, hence $x^{2}\notin I$.

 Notice that  $r_{I}({\bar x}+i\cdot {\bar x}^{2}, {\bar x}+j\cdot {\bar x}^{2})=({\bar x}+(j+1)\cdot {\bar x}^{2}, {\bar x}+(i-1)\cdot {\bar x}^{2})$, hence $X_{I}$ has more than $1$ element (since $x^{2}\neq 0$ in $A/I$). 
Let $n$ be the smallest natural number such that $n\cdot {\bar x}^{2}=0$; then $(X_{I}, r_{I})$  is isomorphic to the solution $(Y, \tilde {r})$.
\end{proof} $\square $

As an application of Proposition \ref{last} we obtain:  
\begin{theorem}\label{Mike} 
Let $ (X,r)$ be a solution of a finite multipermutation level. Suppose that $(X,r)$ is the retraction of a finite solution whose permutation group is nilpotent. If the cardinality of $X$ is larger than $1$ then  there exists a non-trivial braided vector space of set-theoretic type $(X, r^{D})$ for some $D=\{d_{x,y}\}_{x,y\in X}$.
\end{theorem}
\begin{proof} If $(X,r)$ is decomposable the result follows from Proposition \ref{3}. If the solution $(X,r)$ is indecomposable then 
$(X_{I}, r_{I})$ is indecomposable, for any ideal $I$.  By Proposition \ref{last}, any locally monomial BVS can be lifted from
 $(X_{I}, r_{I})$ to $(X,r)$. The result now follows
 from Proposition \ref{stronglynilpotent} and Example 
\ref{simple} in Section $5$.  
\end{proof} $\square $

 \begin{question}
 Is it possible to construct a non-trivial braided vector space of set-theoretic type $(X, r^{D})$  for every finite, involutive, non-degenerate set-theoretic solution $(X,r)$  of a finite multipermutation level? 
 Especially of a multipermutation level $2$?
\end{question}

\section{ $R$-matrices associated to involutive set-theoretic solutions of the  QYBE}\label{singular}

   In \cite{Ok}, Okni{\' n}ski remarked that it would be interesting to know what types of $R$-matrices correspond to  non-degenerate, involutive, set-theoretic solutions of the quantum Yang-Baxter equation. 
 We answer this question below.

\begin{proposition}\label{matrix} Let  $(X, r)$ be a 
 set-theoretic solution of the QYBE. Let $A$ be the matrix associated to $(X,r)$ as in  Definition \ref{monmat}. The following is equivalent:
 \begin{itemize}
 \item[1.] The solution $(X, r)$ is involutive and non-degenerate. 
\item[2.] If we divide $A$ in a natural way into $n$ by $n$ blocks then each  block has exactly one entry equal to  $1$, and all the other entries  zero (for $n=3$ this decomposition is illustrated on matrix $A$ in Examples \ref{A3}, \ref{A4} and \ref{useful}).
  Moreover, $A$ is a  permutation matrix which is symmetric, and $A^{2}$ is the identity matrix.
In particular, $A$ is unitary.  \end{itemize}
\end{proposition}
\begin{proof} Denote $X=\{x_{1}, \ldots , x_{n}\}$. By Definition \ref{monmat}, every column of $M$  has only one nonzero entry equal to one, and  
 the $a_{i,j}^{k,l}$ entry of $A$ is nonzero if and only if $r(x_{i}, x_{j})=(x_{k},x_{l})$. 
  By assumption $r$ is an involutive solution then $r(x_{i}, x_{j})=(x_{k},x_{l})$ implies $r(x_{k}, x_{l})=(x_{i},x_{j})$, hence the $a_{i,j}^{k,l}$ entry of $A$ is nonzero if and only if the $a_{k,l}^{i,j}$ entry of $A$ is nonzero (and it happens exactly  when $r(x_{i}, x_{j})=(x_{k},x_{l})$). It follows that $A$ is a symmetric matrix. Moreover,  $A^{2}$ equals the identity matrix, so $A$ is a permutation matrix.
 It follows that every row in $A$  has exactly one nonzero entry.

 By Corollary 2.3 from \cite{20} and Corollary $8.2.4$ from \cite{21}, 
if $X$ is finite, then an involutive solution  $(X,r)$  is right non-degenerate
if and only if it is left non-degenerate.

It remains to consider the property that $r$ is right non-degenerate. 
 Assume that the solution $(X,r)$ is right non-degenerate.

 Fix element $i$; then we know that, for every $j$, there are $t_{j}, q_{j}\leq n$ such that  $r(x_{i}, x_{j})=(x_{t_{j}},x_{q_{j}} )$.
 Recall that right non-degeneracy of $r$ means exactly  
 that among elements $t_{1}, t_{2}, \ldots ,t_{n}$ we can find every positive integer not exceeding $n$. 
 For $i=1$ the above property implies that  (because $A$ is written in a basis $x_{i}\otimes x_{j}$ with lexicographical order)
 if the first $n$ columns of $A$ are divided into $n\times n$ blocks then each block has a nonzero entry (and since each column has only one nonzero entry, then each such block has exactly one nonzero entry).
By applying it for different $j$, we get that when $A$ is divided in the natural way into n by n blocks, then each block has exactly one  entry equal to $1$, and all other entries equal to $0$.

 By repeating the above reasoning in the reverse order we see that every matrix
 satisfying assumption $2$ gives rise to a non-degenerate involutive set-theoretic solution of the QYBE. 
\end{proof} $\square $

\begin{proposition}\label{unitary} Let $X=\{1, 2, \ldots ,n\}$.
Let $(X,r)$ be a non-degenerate, involutive,  set-theoretic solution of the quantum  Yang-Baxter equation and  let $(X,r^{D})$ be a braided vector space of set-theoretic type where $D=\{d_{i,j}\}_{i,j\in X}$ for some $d_{i,j}\in \mathbb C.$  Let $M$ be the $R$-matrix associated to $(X, r^{D})$ as in Definition \ref{monmat}.
  Then the singular values of the $M$ are $|d_{i,j}|.$   If $|d_{i,j}|=1$ for all $i,j\leq m$ then $M$ is a unitary matrix.
 \end{proposition}
\begin{proof} By Definition \ref{monmat}  the entry  at the intersection of the column indexed by $[i,j]$ 
 and the row indexed by $[k,l]$  of $M$ equals $d_{{i}, {j}}$ if $r({i}, {j})=({k}, {l})$ and $0$ otherwise. 
 Let $M^{*}$ denote the conjugate transpose of  $M$. 
 Notice that 
 $M\cdot  {M}^{*}$ is a diagonal matrix with entries $|d_{i,j}|^{2}$ on the diagonal (in some order). This implies that $M$ is unitary, provided that 
$|{d}_{i,j}|=1$. 
\end{proof} $\square $
\section{ Concrete examples}

In this section we provide some  examples  of some of the  $R$-matrices and  
locally monomial BVS constructed in this paper. 

\medskip
\begin{example}\label{A1}
 This example gives the $R$-matrix obtained in Example \ref{simple} in the case when  $n=4$.
\[
A =\left(
\begin{array}{cccc}
d  E_{44}  & a  E_{1,4}  & b  E_{2,4}  & c   E_{3,4}  \\
c  E_{4,1}    & d  E_{1,1}   & a  E_{2,1}   & b E_{3,1}   \\
b  E_{4,2}  & c  E_{1,2}   & d E_{2,2}  & a E_{3,2}  \\
a E_{4,3}   & b  E_{1,3}   & c E_{2,3}  & d E_{3,3}   
\end{array}\right).
\]

Here $E_{i,j}$  is the $4 \times 4$ matrix which  has all zeros entries  except the element $(i,j)$ which  equals $1$.
\end{example}
\begin{example}\label{A3}
 This example gives the $R$-matrix obtained in Example \ref{simple} in the case when  $n=3$.
\[
A =\left(
\begin{array}{ccc|ccc|ccc}
0 &	0 &	0 &	0 &	0 &	a &	0 &	0 &	0 \\
0 &	0 &	0 &	0 &	0 &	0 &	0 &	0 &	c \\
0 &	0 &	b &	0 &	0 &	0 &	0 &	0 &	0 \\
\hline
0 &	0 &	0 &	b  &	0 &	0 &	0 &	0 &	0 \\
0 &	0 &	0 &	0 &	0 &	0 &	a &	0 &	0 \\
c &	0 &	0 &	0 &	0 &	0 &	0 &	0 &	0 \\
\hline
0 &	0 &	0 &	0 &	c &	0 &	0 &	0 &	0 \\
0 &	0 &	0 &	0 &	0 &	0 &	0 &	b &	0 \\
0 & a &	0 &	0 &	0 &	0 &	0 &	0 &	0
\end{array}\right).
\]

We see that 
\[
A =\left(
\begin{array}{ccc}
b E_{3,3}   & a E_{1,3}   &  c E_{2,3}    \\
c  E_{3,1}  & b E_{1,1}  & a E_{2,1}      \\
a  E_{3,2} &  c E_{1,2}   & b E_{2,2}     \\
\end{array}\right).
\]
\end{example}
\medskip
\begin{example}\label{A2}
  This example gives the $R$-matrix obtained in Example \ref{hura5}.
\[
A =\left(
\begin{array}{cccc}
c E_{4,4}   & c  E_{3,4} &  a E_{1,3}   & b  E_{2,3}  \\
c  E_{4,3}  & c E_{3,3}  & b  E_{1,4}  & a E_{2,4}  \\
b E_{3,1} & a E_{4,1}   & c E_{2,2} & c  E_{1,2}  \\
a E_{3,2} & b  E_{4,2}   & c E_{2,1}   & c E_{1,1}  
\end{array}\right).
\]
 Here $E_{i,j}$ are defined as in Example \ref{A1}.
\end{example}

\medskip

\begin{example}\label{A4}
 This example gives the $R$-matrix obtained as in Proposition \ref{3}  in the case when the cardinality of $X$ is $3$, 
and $X$ has $2$ orbits.
\[
A =\left(
\begin{array}{ccc|ccc|ccc}
0 &	0 &	0 &	0 &	a &	0 &	0 &	0 &	0 \\
0 &	a  &	0 &	0 &	0 &	0 &	0 &	0 &	0 \\
0 &	0 &	0 &	0 &	0 &	0 &	b &	0 &	0 \\
\hline
0 &	0 &	0 &	a &	0 &	0 &	0 &	0 &	0 \\
a &	0 &	0 &	0 &	0 &	0 &	0 &	0 &	0 \\
0 &	0 &	0 &	0 &	0 &	0 &	0 &	b &	0 \\
\hline
0 &	0 &	c &	0 &	0 &	0 &	0 &	0 &	0 \\
0 &	0 &	0 &	0 &	0 &	c &	0 &	0 &	0 \\
0 &	0 &	0 &	0 &	0 &	0 &	0 &	0 &	d
\end{array}\right).
\]
\end{example}

\begin{example}\label{useful} 
 This example  produces the  $R$-matrix obtained as in Proposition \ref{3}  in the case when the cardinality of $X$ is $n$  
and when $X$ has $n$ orbits; the obtained matrix is $n^{2}\times n^{2}$.  The method  uses matrix $S$ which  is the $n\times n$ matrix with consecutive entries and  a matrix $D=diag(d_1, d_2, \ldots,d_n)$ where $d_{1}, \ldots , d_{n}$ are arbitrary. The method is  analogous  for distinct $n$, so 
 we present  it in the case when $n=3$. 
 We will obtain the $R$-matrix obtained as in Proposition \ref{3}  in the case when the cardinality of $X$ is $3$, 
and when $X$ has $3$ orbits. Let $n=3$ and $D=diag(d_1, d_2, \ldots,d_9)$.  
  Let 
\[S=\left(
\begin{array}{ccc}
   1   &   2  &    3\\
     4   &  5  &   6 \\
     7  &    8   &   9
\end{array}\right).
\]
Let  $s_1, s_2, s_3$  denote the columns of $S$, i.e.   
$s_1=(1,4,7)^T$, $s_2=(2,5,8)^T$ and $s_3=(3,6,9)^T$. 
We define a permutation vector  $p=(p_1, p_2, \ldots, p_9)^{T}$,  taking   $p=(s_1^T, s_2^T, s_3^T)^{T}$.
We see that  $p =(1,  4, 7, 2 , 5, 8, 3, 6, 9)^T$. 
Then we form  $A=D P$, where $P$ is  the permutation matrix $P=(e_{p_1}, \ldots, e_{p_9})$ (we permute the columns of the $9$ times $9$ identity matrix).
Denote $d=(d_1, d_2, \ldots,d_9)^{T}$.
We have $A=(d_1 e_1, d_4 e_4, d_7 e_7,  d_2 e_2, d_5 e_5, d_8 e_8,  d_3 e_3, d_6 e_6, d_9 e_9)$ and
\begin{equation}\label{Mama1}
A =A(d)= \left(
\begin{array}{ccc|ccc|ccc}
d_1 &    0 &    0 &  0  &  0 &    0 &    0 &    0 &    0 \\
 0 &    0 &    0 &    d_2 &    0 &    0 &    0 &    0  & 0\\
 0 &    0 &    0 &    0 &    0 &    0 &    d_3 &    0  & 0\\
\hline
0 &    d_4 &    0 &  0  &  0 &    0 &    0 &    0 &    0 \\   
 0 &    0 &    0 &    0 &    d_5 &    0 &    0 &    0  & 0\\
 0 &    0 &    0 &  0  &  0 &    0 &    0 &    d_6 &    0 \\
\hline
   0 &    0 &    d_7 &  0  &  0 &    0 &    0 &    0 &    0 \\     
     0 &    0 &    0 &  0  &  0 &    d_8 &    0 &    0 &    0 \\ 
   0 &    0 &    0 &  0  &  0 &    0 &    0 &    0 &    d_9 
\end{array}\right).
\end{equation}

\medskip

Notice that if $d_2= d_3=d_4=d_6=d_7=d_8=0$  then $A$ is a diagonal matrix.           

If $d_1\ge d_2 \ge \ldots  d_{n^2} > 0$, then the matrix $A$ constructed above is an R-matrix with prescribed singular values $d_1, d_2, \ldots, d_{n^2}$. We recall that the singular values of $A$ are the positive square roots of the eigenvalues of $A^{*} A$, and they play a significant role in many  practical applications. 
If $|d_i|=1$ for $i=1, 2, \ldots, n^2$, then $A$  is unitary.
\end{example}
Now we present   another example of  a unitary  R-matrix.
\begin{example}\label{8}
Let $n=3$ and  define $H=I - \alpha e e^T$, where $e=(1,1,\ldots, 1)^T \in R^9$, $\alpha=2/n^2$.
$H$ is called a reflection (Householder's transformation). 
We can check that $H$ is not an R-matrix, but if we apply the same permutation $P$ as in Example \ref{useful} then  we get that  $A_1=HP$ is a unitary R-matrix.

We have 
\[
A_1 =\frac{1} {9} \left(
\begin{array}{ccc|ccc|ccc}
 7 &    -2  &    -2 &  -2  &  -2 &    -2 &    -2 &    -2 &    -2 \\
 -2 &    -2 &    -2 &    7 &    -2 &    -2 &    -2 &    -2  & -2 \\
 -2 &    -2 &    -2 &    -2 &    -2 &    -2 &    7 &    -2  &  -2\\
\hline
-2 &     7 &    -2 &  -2  &  -2 &    -2 &    -2 &    -2 &    -2 \\   
 -2 &    -2 &    -2 &    -2 &    7 &    -2 &    -2 &    -2  & -2\\
 -2 &    -2 &    -2 &  -2  &  -2 &    -2 &    -2 &    7 &    -2\\
\hline
   -2 &    -2 &    7 &  -2  &  -2 &    -2 &    -2 &    -2 &    -2 \\     
     -2 &    -2 &    -2 &  -2  &  -2 &    7 &    -2 &    -2 &    -2 \\ 
   -2 &    -2 &    -2 &  -2  &  -2 &    -2 &    -2 &    -2 &     7
\end{array}\right).
\]

Notice that $A_1=P- \alpha E$, where $E$ is a matrix with all entries  equal $1$. 
Let $P(3 \times 3)$ be  a Vandermonde matrix  formed for roots of unity, i.e.  
\[
P =\left(
\begin{array}{ccc}
1  &  1   &  1    \\
w_3  & w_2 & w_1 \\
w_3^2 & w_2^2 & w_1^2    \\
\end{array}\right),
\]
where $w_j=\omega^j$  and $\omega=e^{\frac{2 \Pi i}{3}}$. Then we can check that 
$(P\otimes P)^{-1} A_1 (P\otimes P)=A(d)$, where $d=(-1,1,1, \ldots,1)^T$ and $A=A(d)$ is defined in (\ref{Mama1}). 
Note that the columns of $P$ form the set of orthogonal eigenvectors of the ones matrix of size $3 \times 3$. We have $P^{*} P= 3 I$. We obtain that $A_{1}$ is a locally monomial matrix. 
\end{example}

\begin{example}\label{A6} 
Let $E$ be a matrix where all entries  equal $1$. Let $B$  be a locally monomial BVS which was obtained in  Example \ref{simple}, 
with the additional assumption that all non-zero entries of $B$ are equal  to $1$.  Assume 
that $B$ and $E$ are $n^{2}$ by $n^{2}$ matrices for some $n$, and $\alpha , \beta \in {\mathbb C}$ then 
 the matrix  $\alpha B+\beta E$  satisfies the QYBE. 
For example, we obtain  the following unitary $R$-matrix when $\alpha = \frac{-2} {9}$ and $\beta = \frac{7} {9}$.
 \[
A_2 =\frac{1} {9}\left(
\begin{array}{ccc|ccc|ccc}
-2 &	-2 &	-2 &	-2 &	-2 &	7 &	-2 &	-2 &	-2 \\
-2 &	-2 &	-2 &	-2 &	-2 &	-2 &	-2 &	-2 &	7 \\
-2 &	-2 &	7 &	-2 &	-2 &	-2 &	-2 &	-2 &	-2 \\
\hline
-2 &	-2 &	-2 &	7 &	-2 &	-2 &	-2 &	-2 &	-2 \\
-2 &	-2 &	-2 &	-2 &	-2 &	-2 &	7 &	-2 &	-2 \\
7 &	-2 &	-2 &	-2 &	-2 &	-2 &	-2 &	-2 &	-2 \\
\hline
-2 &	-2 &	-2 &	-2 &	7 &	-2 &	-2 &	-2 &	-2 \\
-2 &	-2 &	-2 &	-2 &	-2 &	-2 &	-2 &	7 &	-2 \\
-2 &	7 &	-2 &	-2 &	-2 &	-2 &	-2 &	-2 &	-2
\end{array}\right).
\]
 Surprisingly, this matrix is locally monomial. 
Using the same similarity matrix $P$  as in Example \ref{8}, we find that $(P\otimes P)^{-1} A_2 (P\otimes P)=A(d)$, where  $A=A(d)$ is defined in (\ref{Mama1}), and 
$d=(-w_3,w_2,w_1,w_1,w_3,w_2,w_2,w_1,w_3)^T$.
\end{example}

\medskip

\section*{Acknowledgements}  
 The authors would like to thank  Cesar Galindo, Eric Rowell, Leandro Vendramin,  Robert Weston and Harry Braden
 for their  helpful and enlightening comments. The first author was supported by ERC Advanced grant 320974. We would also like to thank the paper's anonymous referee for providing many useful suggestions as to how to improve the original version.

\end{document}